\documentclass[10pt]{amsart}

\usepackage[usenames,dvipsnames]{xcolor}
\usepackage[bookmarks,colorlinks=true,citecolor=OliveGreen,linkcolor=RoyalBlue]{hyperref}   
\usepackage{amssymb,amsthm}
\usepackage[inline]{enumitem}
\usepackage{mathrsfs}
\usepackage{mathtools}

\usepackage{xparse}

\usepackage{soul}

\usepackage[T1]{fontenc} % Allows diacritics to be copied from the pdf properly
\usepackage[utf8]{inputenc} % Allows diacritics to be included in the source

\usepackage{mathabx}
\usepackage{microtype} 
\usepackage{thmtools}
\usepackage[capitalise]{cleveref}		% Use \cref without environment name

\usepackage{bm}

\usepackage{lmodern}

\makeatletter
\def\thmt@refnamewithcomma #1#2#3,#4,#5\@nil{%
  \@xa\def\csname\thmt@envname #1utorefname\endcsname{#3}%
  \ifcsname #2refname\endcsname
    \csname #2refname\expandafter\endcsname\expandafter{\thmt@envname}{#3}{#4}%
  \fi
}
\makeatother
\declaretheorem[numberwithin=section]{theorem}
\declaretheorem[sibling=theorem]{proposition}
\declaretheorem[sibling=theorem]{corollary}

\declaretheorem[sibling=theorem]{conjecture}

\declaretheorem[sibling=theorem]{lemma}
\declaretheorem[sibling=theorem,style=definition]{definition}

\declaretheorem[sibling=theorem,style=remark]{remark}

\newcounter{claimCounter}[theorem]

\newcounter{subClaimCounter}[claimCounter]

\newcounter{MainSplitClaim}

\newlist{equivalent}{enumerate}{1}
\setlist[equivalent,1]{label=\textup{(\arabic*)}}

\newlist{sublemma}{enumerate}{1}
\setlist[sublemma,1]{label=\textup{(\alph*)}}

\newlist{sublemma*}{enumerate*}{1}
\setlist[sublemma*,1]{label=\textup{(\alph*)},afterlabel=\hspace{5pt}}

\newlist{orderedlist}{enumerate}{1}
\setlist[orderedlist,1]{label=\textup{(\roman*)}}

\newlist{orderedlist*}{enumerate*}{1}
\setlist[orderedlist*,1]{label=\textup{(\roman*)},afterlabel=\hspace{3pt}}

\newcommand{\smallseq}[1]{\langle{#1}\rangle}

\newcommand{\set}[2]{\left\{#1 \,:\, #2\right\}}
\newcommand{\rest}[1]{\restriction{\,{#1}}} 
\newcommand{\tth}{{}^{\textup{th}}}
\newcommand{\conc}{\hat{\,\,}}

\newcommand{\andd}{\,\,\,\&\,\,\,}

\renewcommand{\setminus}{\smallsetminus}

\DeclareMathOperator{\converge}{\downarrow}
\DeclareMathOperator{\diverge}{\uparrow}

\newcommand{\emptystring}{{\lambda}}

\renewcommand{\setminus}{\smallsetminus}

\newcommand{\w}{\omega}
\newcommand{\s}{\sigma}
\newcommand{\vphi}{\varphi}
\renewcommand{\epsilon}{\varepsilon}

\renewcommand{\le}{\leqslant}
\renewcommand{\ge}{\geqslant}

\renewcommand{\preceq}{\preccurlyeq}
\renewcommand{\succeq}{\succcurlyeq}

\newcommand{\Nat}{\mathbb N}

\newcommand{\Tur}{\textup{\scriptsize T}}

\newcommand{\fin}{\textup{\texttt{fin}}}

\newcommand{\lowtwo}{\ensuremath{\textup{low}_2}}

\newcommand{\outside}[1]{{#1}^\complement}
\newcommand{\domin}{\vphi}
\newcommand{\Left}[2]{#1<_L#2}

\DeclareDocumentCommand{\SSS}{ m g}
{
\IfNoValueT{#2}{S(#1)}
\IfNoValueF{#2}{S(#1)_{#2}}
}

\DeclareDocumentCommand{\Y}{ m g}
{
\IfNoValueT{#2}{Y(#1)}
\IfNoValueF{#2}{Y(#1)_{#2}}
}

\newcommand{\seqell}[1]{\bar\ell(#1)}
\newcommand{\oneell}[2]{\ell(#1)_{#2}}

\begin{document}
	
% \makeatletter
% \@namedef{subjclassname@2010}{\textup{2010} Mathematics Subject Classification}
% \makeatother
% \subjclass[2010]{Primary 03D30; Secondary 68Q30, 03D32}

\title{Low$_2$ computably enumerable sets have hyperhypersimple supersets}

\author{Peter Cholak, Rodney Downey and Noam Greenberg}
\address{{\bf Peter Cholak} Mathematics Department,
University of Notre Dame, South Bend, IN 46556 USA}
\email{Peter.Cholak.1@nd.edu}
\address{{\bf Rodney~Downey and Noam Greenberg} School of Mathematics and Statistics,
Victoria University, P.O. Box 600, Wellington, New Zealand.}
\email{rod.downey@vuw.ac.nz, noam.greenberg@vuw.ac.nz}

\thanks{The authors thank  the Marsden Fund of New Zealand.}

\begin{abstract} 
A longstanding question is to characterize the lattice of
supersets (modulo finite sets), $\+L^*(A)$,  of a low$_2$ computably enumerable 
(c.e.) set. The conjecture is that
$\+L^*(A)\cong {\mathcal E}^*$.
In spite of claims in the literature, this longstanding question/conjecture
remains open. We contribute to this problem by solving one of the main 
test cases. We show that if c.e.\ $A$ is low$_2$ then $A$ has an atomless
 hyperhypersimple superset. 
In fact, 
  if $A$ is c.e.\ and low$_2$, then for any 
 $\Sigma_3$-Boolean algebra~$B$ there is some c.e.\ $H\supseteq A$ such that $\+L^*(H)\cong B$.
\end{abstract}

\maketitle

%%%%%%%%%%%%%%%%%%%%%%%%%%%%%%%%%%%%%%%%%%%%%%%%%%%
%%%%%%%%%%%%%%%%%%%%%%%%%%%%%%%%%%%%%%%%%%%%%%%%%%%
\section{Introduction}
%%%%%%%%%%%%%%%%%%%%%%%%%%%%%%%%%%%%%%%%%%%%%%%%%%%
%%%%%%%%%%%%%%%%%%%%%%%%%%%%%%%%%%%%%%%%%%%%%%%%%%%

The concern of this paper is the interplay of two fundamental
structures arising in classical computability, the lattice of
computably enumerable sets and Turing reducibility $\le_T$.
 Turing \cite{Turing:39} defined $A\le_T B$ 
as the most general notion of reducibility between 
problems coded as sets of non-negative integers.
Reducibilities 
give rise to partial orderings which 
measure relative computational complexity.
In the classic paper \cite{P}, Post suggested that the study of
the lattice of computably (recursively) enumerable (c.e.) sets was
fundamental in computability theory.  
In the words of Soare \cite[p.\:viii]{Soa}
\begin{quote}
``Post \cite{P} stripped away the formalism associated with the development
of recursive function theory in the 1930's and revealed in a clear informal 
way the essential properties of r.e.\ sets and their role in 
G{\"o}del's incompleteness theorem''.
\end{quote}
Computably enumerable sets are the halting sets of Turing machines, i.e.,
 the domains of partial computable functions. They thus 
  represent natural ``semi-decidable'' problems such as 
instances of Hilbert's 10th problem, the Entschiedungsproblem, 
the Post Correspondence Problem,  sets of 
consequences of formal systems, and many others.

The interplay of these two basic objects, Turing reducibility and c.e.\ sets, 
has a long and rich history.
The computably enumerable sets under union and intersection
form a lattice, denoted by ${\mathcal E}$, and 
a common object of study is
${\mathcal E}^*$, which is ${\mathcal E}$ modulo 
the congruence $=^*$, where $A=^*B$ means that the symmetric 
difference of $A$ and $B$ is finite.  The Turing degrees of  c.e.\ sets 
form
an upper semilattice, denoted by ${\mathcal R}$.  Ever since the
groundbreaking paper of Post, there has been a persistent intuition
that structural properties of computably enumerable sets have
reflections in their degrees, and vice versa.  In particular, {\em
  definability} in ${\mathcal E}$ should be linked with {\em
  information content} as measured by Turing reducibility. 

The simplest possible illustration of this is the fact that the
complemented members of ${\mathcal E}$ are exactly the members of
${\bf 0}$, the degree of the computable sets.  One of the main avenues of attack has been to link 
properties of c.e.\ sets with the jump operator, $A'=\{e\mid \Phi_e^A(e)\converge\}$, the halting problem \emph{relative to} $A$.
The jump operator gives one way of understanding information content of 
members of ${\mathcal R}$. A deep example is Martin's result \cite{mart} that the Turing degrees of maximal sets are
exactly the high c.e.\ Turing degrees. Here a c.e.\ set $A$ is \emph{high} if $\emptyset''\equiv_\Tur A'$, i.e., if the jump operator does not distinguish between~$A$ and the halting problem $\emptyset'$. A coinfinite c.e.\ set~$A$ is \emph{maximal} if it represents a co-atom in ${\mathcal E^*}$, that is, if for every c.e.\ set $W\supseteq A$, either $A =^* W$ or $A=^* \Nat$. 

On the other hand, we call $A$ \emph{low} if $A'\equiv_\Tur \emptyset'$, that is, if the jump operator does not distinguish between~$A$ and the computable sets. Such sets have relatively little information content, so we would guess that low sets should perhaps share some properties with the computable sets. 

This intuition has seen many realizations in the study of the Turing degrees of low c.e.\ sets. One such example is  
Robinson's \cite{Rob} result that 
if we have computably enumerable sets $B\le_T A$ with $B$ 
low, then we can split $A=A_1\sqcup A_2$ as  two c.e.\ sets with 
$A_1\oplus B$ Turing incomparible with $A_2\oplus B$. This result shows that
splitting and density can be combined above low c.e. degrees,
something that, famously, Lachlan \cite{monster}
showed {fails} in general.

How is this intuition reflected in the properties of 
low c.e.\ sets in ${\mathcal E}^*$? 
Soare \cite{auto2} showed that if 
$A$ is low then 
$\+L^*(A)\cong {\mathcal E}^*$. 
Here $\+L^*(A)$ denotes the lattice of c.e.\ supersets of $A$ modulo~$=^*$.
Thus it is impossible to distinguish~$A$ from~$\emptyset$ using properties 
of its lattice of supersets.
We remark that  Soare proved  
that the isomorphism is \emph{effective} in that
there are computable function $f$ and $g$
with $W_e\mapsto W_{f(e)}$ (from $\+L^*(A)$ to 
${\mathcal E}^*$) and $W_e\mapsto W_{g(e)}$
for the return map,
inducing the isomorphism.

Note that Martin's Theorem says that if a c.e.\ degree is high then
then it contains a maximal set~$A$.
If~$A$ is maximal than   $\+L^*(A)$ is the  two-element lattice --- as far as possible from $\+E^*$. It is natural to wonder what level of computational
power for degrees stop c.e.\ sets 
with those degrees from 
resembling computable sets. 
A natural demarcation seems to be at low$_2$. Recall that a set $A$ is low$_2$ if the \emph{double} jump does not distinguish between~$A$ and the computable sets: $A''\equiv_\Tur \emptyset''$.\footnote{Indeed the reader should
  recall that, more generally, a set $A$ is low$_n$ if $A^{(n)}$, the $n\tth$ jump of 
$A$, is as low as possible: $A^{(n)} =
  \emptyset^{(n)}$, equivalently, $\Delta_{n+1} = \Delta^A_{n+1}$; and a $\Delta_2$ set~$A$ is
  high$_n$ if its $n\tth$ jump is as high as possible: $A^{(n)} = \emptyset^{(n+1)}$, equivalently, $\Delta_{n+2} =
  \Delta^A_{n+1}$.}  
Shoenfield \cite{Sho} proved that if a c.e.\ Turing degree is 
not low$_2$, then it contains a c.e.\ set~$B$ that has {no} maximal superset.
Hence, in particular, $\+L^*(B)\not \cong {\mathcal E}^*$.
It follows that if there is a collection~${\mathcal J}$ of degrees characterized 
by their jumps, such that 
 deg$(A)\in {\mathcal J}$ implies $\+L^*(A)\cong {\mathcal E}^*$, then 
${\mathcal J}$ must be a subclass of the low$_2$ degrees.

Maass \cite{ma} extended Soare's result to prove that 
for a c.e.\ set~$A$, 
$\+L^*(A)$ is {effectively} isomorphic to ${\mathcal E}^*$ 
if and only if $A$ is semilow$_{1.5}$. Here a c.e.\ set $A$ 
is called semilow$_{1.5}$ if 
$\{e:W_e\subseteq^* A\}$ is~$\Sigma_2$.
Being semilow$_{1.5}$
 is a ``pointwise'' variation of lowness. It is known that
 a c.e.\ degree that is not low contains a c.e.\ set 
which is not semilow$_{1.5}$.  In particular, there are low$_2$ 
sets which are not semilow$_{1.5}$ (Downey, Jockusch and Schupp \cite[Theorem 1.5]{DJS}).

The following conjecture has been around for some time.

\begin{conjecture}[Soare and others]
\label{conj} If $A$ is a coinfinite $\lowtwo$ c.e.\ set then
$\+L^*(A)\cong {\mathcal E}^*$.
\end{conjecture}

This conjecture is stated as an open question
in Soare \cite{Soa}.
We remark that Conjecture \ref{conj} has been claimed  as a theorem 
in several places, notably as being due to 
Harrington, Lachlan, Maass
and Soare as stated in  Harrington-Soare \cite[Theorem 
1.8]{HarringtonSoare96}, as well as in lectures by Soare in the late 1990's. But despite the best efforts of 
computability-theorists in the last 30 or so years, no proof has 
appeared,
so it is surely an open question by this stage.

Notice that for this conjecture to hold, we must need some other method 
of constructing the isomorphism than an effective isomorphism. 
Likely if Conjecture
\ref{conj} is true, then methods along the lines of the powerful 
$\Delta_3$ methods
introduced by Soare and Harrington \cite{HarringtonSoare96}
or Cholak \cite{Cholak:95} will likely be needed.
Cholak \cite{Cholak:95} 
gave some partial evidence for the validity of the conjecture 
by showing that every semilow$_2$ c.e.\ set with 
the \emph{outer splitting 
property} has 
$\+L^*(A)\cong {\mathcal E}^*$. It is not important for our story 
what these properties are, save to say that
they give weaker guessing 
methods than semilow$_{1.5}$-ness, but which are sufficient given the $\Delta_3$ isomorphism method. In this paper we make a modest contribution supporting Conjecture \ref{conj}.
We will soon explain what we mean by guessing, and how our work fits into the
programme of establishing the conjecture.

Lachlan \cite{Lachlan2}
proved that if $A$ is low$_2$ then $A$ has a maximal superset (i.e., 
$\+L^*(A)$ contains a co-atom). This is currently the strongest 
general result about all low$_2$ c.e. sets.
It is likely not too difficult to modify Lachlan's technique
to extend this result to $k$-quasimaximal sets, c.e.\ sets which are the 
intersections of exactly $k$ maximal sets, by simultaneously constructing
$k$ maximal supersets $M_1,\dots,M_k$ 
of $A$ such that $M_i\neq^*M_j$ for $i\ne j$. If $Q$ is $k$-quasimaximal
then $\+L^*(Q) $ is a $k$-atom boolean algebra.

Lachlan \cite{Lachlan:68} 
proved that  $\+L^*(A)$ is a Boolean algebra 
if and only if it is {hyperhypersimple}.
Hyperhypersimple sets  sets were defined, if not constructed, by Post \cite{P}.
Recall that a coinfinite c.e.\ set $A$ is called 
\emph{simple} if its complement ${A^\complement} = \Nat\setminus A$ contains no 
infinite c.e.\ subsets.
A much stronger property is that $A$ is \emph{hyperhypersimple},
which is defined
as there being  is no collection of infinitely many finite, pairwise disjoint, uniformly c.e.\ sets, such that each element of the collection intersects the complement of~$A$. However, Lachlan's characterisation in terms of Boolean algebras is used more often. A hyperhypersimple set~$A$ is \emph{atomless} if $\+L^*(A)$ is the atomless Boolean algebra. In a talk in 2006, 
Cholak pointed out the next test case for the low$_2$ conjecture: atomless
hyperhypersimple sets.
In this paper we affirm Cholak's conjecture.

\begin{theorem} \label{thm:atomless}
  Every coinfinite $\lowtwo\!$ c.e.\ set has an atomless hyperhypersimple superset. 
\end{theorem}

We now  try to explain why this is an important test case for the full 
low$_2$ conjecture. 
By necessity, this explanation is somewhat technical as it 
revolves around issues of $\Delta_3$ guessing and 
arguments akin to those of 
the deep paper \cite{HarringtonSoare96} by Harrington and Soare. 
% The reader should imagine 
% that behind the scenes there is a step by step priority argument.
% This argument will be run on a priority tree where nodes 
% have infinitely many outcomes representing certain guessing procedures 
% which are being used.
% More details of this impressionistic 
% analysis can be found later in the paper where comments will be more in context. 

The reader might recall  Friedberg’s \cite{F58}
construction of a maximal set $M$. For each~$e$, we need to ensure that if 
$W_e\supseteq M$, then either $W_e=^* M$ or $W_e=^*\Nat$, whilst still keeping~$M$ coinfinite.  At each stage $s$ we will have listed in order the elements of the complement $M^\complement_s=\{b_{i,s}\mid i\in \omega\}$. 
To make ${M}$ coinfinite, we make sure that 
for each $e$, $\lim_s b_{e,s}=b_e$ exists.
For the maximality requirement, the idea is to make, for all
$n\ge e$, the $n\tth$ member of the complement of $M$ a member of $W_e$, as follows. Assuming that $b_{e,s},\dots,b_{n-1,s}\in W_{e,s}$ but $b_{n,s}\not \in W_{e,s-1}$,  if we see some $x=b_{m,s}\in W_{e,s}$ for $m\ge n$, we make $x=b_{n,s+1}$ by  enumerating the 
 elements $b_{n,s},\dots,b_{m-1,s}$ into $M_{s+1}$, causing $b_{n,s+1}\in W_{e}.$ If this happens for each $n\ge e$ then $W_e\cup M=^* {\mathbb N}.$ On the other hand, if we fail
to keep finding such $b_{m,s}$, from some point $n$ onwards,
then $W_e\cup M=^* M$.

Of course, we have to worry about differing requirements 
putting things into $M$ since $W_e$ might want to make $b_{n,s}\in W_e$, and some 
other $W_{e'}$ might want to put this into $M$ to cause $b_{n,t}\in W_{e'}$,
so we need some way to reconcile such 
conflicting requests. This is done using Friedberg's priority method. In modern terminology, if $e'>e$, then the requirement dealing with $W_{e'}$ will guess the eventual behaviour of $W_e$, and therefore be able to align with it: either wait for elements of $M^\complement$ to enter~$W_e$, before attmepting to find elements of the complement in $W_{e'}$; or only start acting at a stage after which no more elements of $M^\complement$ enter~$W_e$. To implement this, Friedberg’s brilliant
idea was using \emph{$e$-states}.\footnote{Thus, in effect, the $e$-state method is a forerunner of infinite injury arguments.} The $e$-state of $b$ at stage~$s$ is the binary string that records the indices $d\le e$ such that $b\in W_{d,s}$. 
Friedberg's idea is to put, for each $n\ge e$, $b_{n,s}$ into the 
lexicographically \emph{maximized} $e$-state.
This means that almost all of the complement of~$M$ will be in the same 
$e$-state, and hence be in~$W_e$ or out of~$W_e$. 
In our constructions below, we use the modern tree-of-strategies terminology, rather than explictly using $e$-states. However, the notion of $e$-states appears necessary when constructing isomorphisms of lattices of c.e.\ sets, and so will be useful for the current dicussion. 

Suppose that we try to emulate Friedberg’s method to construct 
a maximal $M\supset A$, where the opponent is controlling the c.e.\ set~$A$. 
It is within the opponent's power to take elements from the complement ${M}^\complement_s$ such as $b_{i,s}$, and declare that they are in $M$, since they are in $A_{s+1}$, and 
$M\supseteq A$. The danger in implementing the strategy above for dealing with~$W_e$ is that we could believe that we see infinitely many elements of $M^\complement$ in $W_e$, and keep dumping elements not yet seen in $W_e$ into~$M$, but then the elements of~$W_e$ go into~$A$. This would result in~$M$ being cofinite. We should think of the elements of~$A$ as ``phantom elements'', ones that shouldn't really be considered, except that we cannot know this in advance. The way around this is for the requirement to \emph{guess} whether there are infinitely many numbers that are in~$W_e$ and not in~$A$, i.e., if $W_e\setminus A$ is infinite. This statement is $\Pi_2(A)$; since~$A $ is $\lowtwo$, this means that this statement is~$\Delta_3$. 

Lachlan’s proof in \cite{Lachlan2} (see also \cite[Theorem XI.5.1]{Soa}\footnote{This is the notorious ``blondes and brunettes'' construction in 
Soare.}) is a reasonably delicate construction that uses a version of $\Delta_3$ guessing. 
Lachlan's original proof seemed {ad hoc} and combinatorial.
As part of our paper we give a new thematic proof of Lachlan's 
Theorem based around 
$\Delta_3$ guessing
on an $\omega$-branching priority tree. As both the ``yes'' and ``no'' answer to the question whether $W_e\setminus A$ is infinite are $\Sigma_3$ (as $\Delta_3 = \Sigma_3\cap \Pi_3$), the outcomes are pairs $(\text{``yes''},n)$ and $(\text{``no''},n)$ where $n$ is the witness for the~$\Sigma_3$ predicate holding. When considering more than one requirement, we are guessing the answers for a Boolean combination of $\Pi_2(A)$ questions; but these too will be $\Delta_3$, so will be subject to the same guessing procedure. 
$\Delta_3$ priority arguments are much less common in computability theory
than the usual infinite injury arguments (such as the Thickness Lemma, 
or the Minimal Pair argument) which can be performed on finitely 
branching trees. Harrington and Soare \cite{HarringtonSoare96}, 
Downey and Greenberg \cite{DGbook}, Shore and Slaman \cite{ShSl1,ShSl2}
are some examples of $\Delta_3$ arguments.

The method of $\Delta_3$ guessing does not appear to be sufficient for results stronger than the construction of a maximal superset. Suppose that we are trying to show that $\+L^*(A)\cong \+E^*$. For each~$e$ and each~$e$-state~$\s$, we need to ensure that infinitely many elements of~$\Nat$ have $e$-state~$\s$ (with respect to the list of c.e.\ subsets of~$\Nat$) if and only if infinitely many elements of~$A^\complement$ have $e$-state~$\s$, with respect to a listing of the c.e.\ supersets of~$A$. The $e$-states correspond to measures
of intersections and non-intersections of c.e.\ sets. 

The method of $\Delta_3$ guessing will allow us to correctly guess which $e$-states we should try to fill. But now we have the more complicated task of actually finding potential elements and enumerating them into the correct sets, so that we can match $e$-states as required. The difference between the maximal set construction and the more general construction is that in the former, for each~$e$, there will be exactly one $e$-state that we need to populate. So there, the only difficulty is in guessing which $e$-state will be populated; once this is decided, all elements of the complement of~$M$ will be enumerated into the same sets (as usual, except for finitely many). If there are more than one $e$-state to populate, say $\s$ and~$\tau$, then given an element $b$ that at a stage~$s$ seems to be in the complement of the set we are building, we need to decide whether to enumerate~$b$ into sets so as to make the $e$-state of~$b$ either~$\s$ or~$\tau$. We need to ensure that both~$\s$ and~$\tau$ will be populated by infinitely many ``true'' elements, elements that are not in~$A$. The question is whether given such~$b$ at stage~$s$, should we believe
that this is the true situation, and, moreover, what do we do 
if we act on false beliefs?

The key to all proofs we mentioned (Soare, Maass, Cholak, etc)
 where $\+L^*(A)$ is shown to be 
isomorphic to ${\mathcal E}^*$ 
is some kind of guessing procedure to understand when elements seem to 
be truly outside~$A$. If $A$ is low, semilow, or even semilow$_{1.5}$ we have a ``pointwise'' 
testing process where we can guess whether individual elements are in~$A^\complement$, and get the answer \emph{more-or-less immediately}. For example, suppose that
 $A$ is low.
Then using,  for example,
the Robinson trick, we can ask $A'$, and hence $\emptyset'$, whether some 
$z\in {A}^\complement_s$ (in some desirable $e$-state) is actually outside~$A$, and 
must eventually get a true ``yes'' if there are infinitely many such 
potential~$z$. In his proof that semilow$_{1.5}$ c.e. sets have lattices of supersets effectively automorphic with ${\mathcal E}^*$ (\cite{ma}),  Maass points out that indeed, it is not enough to know that there are infinitely many elements in some state, but we also need to capture them. Maass's outer splitting property, mentioned above, is an elaboration on the low guessing technique, that works if~$A$ is semilow$_{1.5}$. 

The key problem appears to be that of \emph{splitting} $e$-states. For some $e$-state~$\s$, we have somehow guaranteed that infinitely many true elements (elements outside~$A$) have $e$-state~$\s$. Suppose that~$\Delta_3$ guessing tells us that both $\s\conc 0$ and $\s\conc 1$ are $(e+1)$-states that need ``filling''. In the construction, we are given an infnite set~$E$ of elements in state~$\s$, and we know that infinitely many of these are outside~$A$. The problem is to split~$E$ into two sets~$E_0$ and~$E_1$, while ensuring that both $E_0\setminus A$ and $E_1\setminus A$ are infinite. We would then enumerate the elements of~$E_1$ into the $(e+1)\tth$ c.e.\ superset of~$A$, making them have $(e+1)$-state $\s\conc 1$, while keeping the elements of~$E_0$ outside that set, making them have $(e+1)$-state $\s\conc 0$. This is precisely what is required in the construction of an atomless hyperhypersimple superset. As Cholak pointed out,\footnote{As did Maass in unpublished handwritten notes from a seminar.} the dynamics of Lachlan’s construction do not seem to be modifiable to obtain this. 

In our main construction in \cref{sec:atomless_supersets}, we introduce a novel method for performing this splitting. This method does not rely solely on $\Delta_3$ guessing; we also use domaination properties of $\lowtwo$ sets, namely, the fact that if $A$ is $\lowtwo$, then there is a $\emptyset'$-computable function that dominates all functions computable from~$A$. 

To end this discussion, we should mention why our new technique does not seem to immediately solve the original problem of showing that $\+L^*(A)\cong \+E^*$. The issue is very delicate, and as is often the case, relies on discrepancies in timing. When constructing an atomless hyperhypersimple superset~$H$, we are, in some sense, controlling the supersets of~$H$, which means that all $e$-states have equal status (this is by necessity an imprecise simplification). When constructing an isomorphism, the opponent has within their power to shift elements from one~$e$ state to another. There is a difference between what are colloquially known as ``low'' and ``high'' $e$-states. The opponent can move elements from low to high $e$-states. Unlike the low guessing construction, or the construction relying on the outer splitting property, our splitting method is not immediate: we have to wait an unknown amount of time until we get a certification that certain elements are outside~$A$, and for finitely many elements, such certification may never happen. While we are waiting for certification, the opponent will shift elements outside a desirable low state. With our new splitting method, we are thus able to split high states, but not low states. In the hyperhypersimple construction, all states are high. 

\smallskip

In addition to Lachlan's characterisation of hyperhypersimple sets, he also classified the isomorphism types of the resulting Boolean algebras: he showed that the Boolean algebras $\+L^*(A)$ for hyperhypersimple sets~$A$ are precisely the $\Sigma_3$-Boolean algebras --- quotients of the computable, atomless Boolean algebra by $\Sigma_3$ ideals. We show in this paper that \cref{thm:atomless} can be extended to obtain all such algebras as the lattice of supersets of a superset of any given $\lowtwo$ c.e.\ set.

\begin{theorem} Suppose that $A$ is a coinfinite, $\lowtwo$ c.e.\ set. Let 
$B$ be a $\Sigma_3$ boolean algebra. Then there is a 
c.e.\ superset $H\supseteq A$ with $\+L^*(H)\cong B$.\end{theorem}

\section{Maximal supersets} % (fold)
\label{sec:maximal_supersets}

As mentioned above, in order to present our proof of \cref{thm:atomless}, it would be useful to first give a ``modern'' or at least 
``thematic'' proof of Lachlan's theorem, using $\Delta_3$ guessing 
on a priority tree:

\begin{theorem}[Lachlan \cite{Lachlan2}] \label{thm:Lachlan_maximal}
	Every $\lowtwo\!$ coinfinite c.e.\ set has a maximal superset.
\end{theorem}

\subsection{Discussion}

We are given a coinfinite, $\lowtwo$ c.e.\ set~$A$; we enumerate a maximal c.e.\ set $M\supseteq A$.

This proof follows, to a certain extent, the $\Delta_3$ automorphism machinery of Harrington and Soare \cite{HarringtonSoare96}. The construction is performed on a tree of strategies, similarly to many infinite injury priority arguments. Nodes on the tree represent guesses about the eventual behaviour of some aspects of the construction, and use their guesses to meet a requirement that they are assigned to. In the current construction, nodes of length $e+1$ make guesses about how the $e\tth$ c.e.\ set~$W_e$ interacts with the construction, and attempt to meet the $e\tth$ maximality requirement: either $W_e\cup M =^*\Nat$, or $W_e\cup M=^* M$. We will identify a \emph{true path}, the path of nodes whose guesses are correct; nodes on the true path will be successful in meeting their requirements. 

In some ways, though, the usage of the tree of strategies is quite different to most constructions. For one, at a stage~$s$ of the construction, we do not define a path of ``accessible'' nodes, those whose guesses appear to be correct at that stage. Further, there will not be much explicit interaction between strategies: they will not impose restraint, or cause initialisation of other strategies. We will not use the terminology of relative priority.  

More importantly, we use the tree as the hardware of a \emph{pinball machine}. We view the priority tree as growing downwards; so the root of the tree is the ``top'' node. During the construction, we place some \emph{balls} at the root of the tree. These balls represent numbers that at a stage~$s$ of the construction are not in~$A_s$ (the stage~$s$ approximation to the given c.e.\ set~$A$). At any given stage, there will be only finitely many balls on the machine, but we will ensure that every element of the complement $A^\complement$ of~$A$ is placed on the machine at some stage. Once placed at the root, we allow balls to move between nodes on the tree. Some balls may be enumerated into~$M$, at which time they are removed from the machine. This includes all the elements of~$A$, since $A\subseteq M$. We will ensure that balls that are never enumerated into~$M$ only move during finitely many stages of the construction, so they eventually arrive at a permanent ``resting place''. The main mechanics of the construction are the determination of which balls move, where they move to, and which balls are removed from the machine. 

We remark that the term ``pinball machine'' was earlier used by Lerman, to organise priority arguments (often used for embedding lattices into the c.e.\ degrees). The Harrington-Soare pinball machine, that we use here, is different in that the balls move from the root to other nodes, which is an opposite direction to Lerman's machines. 

As mentioned in the introduction, another aspect of the Harrington-Soare machinery is the employment of $\Delta_3$-guessing. This is the main way that we use, in this construction, the assumption that~$A$ is $\lowtwo$. In a typical $\Pi_2$ argument (such as the minimal pair construction), during the construction we guess the outcome of $\Pi_2$ questions based on what we see at each stage. In the current construction, we ask more complicated questions, namely, $\Pi_2(A)$ questions, and we use the fact that $\Pi_2(A)\subseteq \Delta_3$. Rather than observing the construction and making guesses accordingly, we are given an approximation to the answer to a $\Delta_3$ statement, that may or may not be aligned with what is measured at a given stage; we just know that in the limit, the approximation will give us the correct answer. As we shall detail below, we will rely on the recursion theorem to ensure that we indeed approximate the correct answers. The guessing process for $\Delta_3$ facts is more complicated than that of $\Pi_2$ facts. Namely, to guess membership in a $\Delta_3$ set~$S$, we need to guess an answer ($n\in S$ or $n\notin S$), and to guess an existential witness for the $\Sigma_3$ predicate being guessed. This implies that the tree of strategies needs to be infinite-branching.

\subsubsection{The requirements}

Let us consider now the goals of the construction. As discussed, the balls that are never removed from the machine will be precisely the elements of $M^\complement$. Each such ball will eventually settle as a ``resident'' of some node~$\beta$ on the tree. For a node~$\alpha$ on the tree, we will let $\Y{\alpha}$ denote the collection of balls which are permanent residents of nodes $\beta\succeq \alpha$. The rules about ball movement will ensure that $\Y{\alpha}$ is a d.c.e.\ set. To enter~$\Y{\alpha}$, a ball~$x$ will need to first pass through~$\alpha$, i.e., be a resident of~$\alpha$ at some stage. It may later move to extensions of~$\alpha$. It is possible that it will be removed from $\Y{\alpha}$ by being \emph{pulled} by some node, that lies to the left of~$\alpha$ on the tree. In any case, balls can be removed from $\Y{\alpha}$ by enumerating them into~$M$. We will ensure, however, that $\Y{\alpha}\cup M$ is in fact a c.e.\ set, though not uniformly in~$\alpha$. This is because either only finitely many balls outside~$M$ will ever enter $\Y{\alpha}$, or finitely many balls outside~$M$ ever leave $\Y{\alpha}$, or $\Y{\alpha}= \emptyset$. 

We will ensure the following:

\begin{sublemma}
  \item For each node $\beta$ on the true path, $\Y{\beta}=^* \outside{M}$. That is, all but finitely many elements of $\outside{M}$ will have passed through~$\beta$ at some stage and settled at~$\beta$ or some extension of~$\beta$, i.e., below~$\beta$ (recall that our trees grow downwards). 

  \item For each~$e$, if $\beta$ is the node on the true path of length $e+1$, then $\Y{\beta}$ is either almost contained in, or almost disjoint from, $W_e$. 

  \item We also need to ensure that $M$ is coinfinite. 
\end{sublemma}

For~(a), we will ensure that only finitely many balls have permanent residence to the left of~$\beta$, to the right of~$\beta$, or above~$\beta$. 

For~(b), we will ensure that either all balls that ever pass through~$\beta$ are already seen to be in~$W_e$ when they do, or that we know that only finitely many elements of~$W_e$ will ever be available to pass through~$\beta$. 
In the first case, the node~$\beta$ will ensure that all but finitely many balls outside~$W_e$ will be enumerated into~$M$ (we will say that they are \emph{eliminated} from the machine). 

For~(c), we will ensure that infinitely many nodes~$\beta$ on the true path hold on to balls and ensure that they are not eliminated, i.e., keep them out of~$M$. Similar action will ensure that balls indeed only move finitely much, i.e., they do eventually settle at some node.

\subsection{Setup}

We now go into the details.

\subsubsection{The tree} % (fold)

We let the tree of strategies be the collection of all sequences of the symbols $\fin_n$ and~$\infty_n$. These symbols are called \emph{outcomes}.   
We write $\alpha\preceq \beta$ to indicate that $\alpha$ is a prefix of~$\beta$. 

In addition to extension, we will use the \emph{Kleene-Brouwer} ordering on the tree. To define this we fix an ordering of the outcomes,  say 
\[
  \infty_0 < \fin_0 < \infty_1 < \fin_1 < \cdots.
\]  
We say that $\alpha$ lies to the left of~$\beta$, and write $\Left{\alpha}{\beta}$, if $\alpha$ and~$\beta$ are incomparable (neither is a prefix of the other), and $\alpha(k)< \beta(k)$ for the least~$k$ such that $\alpha(k)\ne \beta(k)$. We write $\alpha\le \beta$ if either $\Left{\alpha}{\beta}$ or $\beta\preceq \alpha$. As this is the only linear ordering of strategies that we will use, we use the notation $\alpha\le \beta$ rather than $\alpha\le_{\textup{KB}} \beta$. 

We let $\emptystring$ denote the root of the tree (the empty sequence).

\subsubsection{Notation for the pinball machine} % (fold)

As discussed, at each stage, finitely many \emph{balls} $x$ will reside at some nodes of the machine. We let
\[
	 \SSS{\alpha}{s}
\]
denote the collection of~$x$ that reside at the node~$\alpha$ at the \emph{beginning} of stage~$s$. 
We let $\Y{\alpha}{s}$ be the collection of all balls that at the beginning of stage~$s$ reside at~$\alpha$ or below~$\alpha$, i.e., at some extension of~$\alpha$:
\[
	\Y{\alpha}{s} =  \bigcup_{\beta\succeq \alpha}  \SSS{\beta}{s}.
\]
So $\Y{\emptystring}{s}$ is the collection of balls that are on the machine at the beginning of stage~$s$.

\subsubsection{True stage, and delayed, enumerations}

Since the given set~$A$ is $\lowtwo$, it is not high. C.e.\ sets that are not high have ``true stage'' enumerations with respect to any given computable growth rate: for any computable function~$f$, there is a computable enumeration $(A_r)$\footnote{Here the notation $(A_r)$ refers to the infinite sequence of finite sets $A_0,A_1,\dots$ that form the enumeration of~$A$.} of~$A$  such that for infinitely many stages~$r$, the first $f(r)$ many elements of the stage~$r$ complement $\outside{A}_r$ are ``correct'': they are elements of the complement $\outside{A}$. See \cite[Lem.\:XI.1.6]{Soa}.\footnote{Note that the term ``true stages'' is sometimes used to refer to the Dekker ``nondeficiency'' stages; we will not use these.}

We fix an enumeration $(\tilde{A}_r)$ of~$A$, that is true with respect to the function $f(r) = r^2$. During the construction, though, it will be useful to \emph{delay} this enumeration. We will define, during the construction, an enumeration $(A_s)$ of~$A$, such that for all~$s$ there is some~$r$ such that $A_s = \tilde A_r$. We will write $n(s)=r$. 

Let~$s$ be a stage, and suppose that $n(s)=r$. We will let the construction run: balls will move around, or be enumerated into~$M$, and it will be useful to keep track of these actions as happening during several stages: $s$, $s+1$, $s+2$, and so on. While we are doing this, we keep our enumeration of~$A$ fixed: that is, $\tilde A_r = A_s = A_{s+1} = A_{s+2} = \cdots$. After finitely many stages, we will reach a stage $t\ge s$ at which we decide that there is nothing we want to do. We then declare~$t$ to be a ``new balls'' stage, and set $A_{t+1} = \tilde{A}_{r+1}$,\footnote{Recall that the index~$s$ denotes the objects as they are at the \emph{beginning} of stage~$s$. During stage~$s$ we perform some actions, resulting in the objects indexed by~$s+1$. Thus, if during stage~$t$ we place new balls on the machine, these balls will be elements of $\Y{\emptystring}{t+1}$, not of $\Y{\emptystring}{t}$.} i.e., set $n(t+1) = r+1$. 

Thus, the function $s\mapsto n(s)$ will be weakly increasing and onto; we will have $n(s+1)\ne n(s)$ exactly when $s$ is a new balls stage, in which case $n(s+1) = n(s)+1$. What is important is that the enumeration $(\tilde A)_r$ is given to us, whereas the enumeration $(A_s)$ is defined by us during the construction: based on how the construction develops, we decide whether to declare a ``new balls'' stage, i.e., when to increase $n$ by~1 and thus enumerate more numbers into~$A$. We start with $n(0)=0$. Observe that for any stage~$s$, $n(s)$ is the number of new balls stages $t<s$. We will ensure that there are infinitely many new balls stages, so that indeed $\bigcup_s A_s = A$. 

\begin{definition} \label{def:Q_s}
 We let 
\[
  Q_s = \{\text{the smallest } n(s)^2 \text{ many elements of }\outside{A}_s\}.
\] 
A stage~$s$ is \emph{$A$-true} if $Q_s \subseteq \outside{A}$.
\end{definition}

Thus, if $s<t$ are successive new balls stages, then $Q_{s+1} = Q_{s+2} = \cdots = Q_t$, and $s+1$ is $A$-true if and only if all the stages $s+1,s+2,\dots, t$ are $A$-true. A stage~$s$ is $A$-true if and only if the stage $r = n(s)$ is true for the enumeration $(\tilde A_r)$ with respect to $r\mapsto r^2$. By assumption, there are infinitely many such~$r$. As we will ensure that $s\mapsto n(s)$ is onto, there will be infinitely many $A$-true stages~$s$. 

We will ensure that at every stage~$s$, $\Y{\emptystring}{s}\setminus M_s = Q_s$. So at a new balls stage~$s$, we will place all the elements of~$Q_{s+1}$ that have not yet been put on the machine, at the root. As stages go by, some of these elements may be enumerated into~$M$. 

The reason for using true stages for the function $f(r) = r^2$ is the following. If~$s$ is a new balls stage then 
 \[
 |Q_{s+1}|  = |Q_s| + 2n(s) +1 > |Q_s| + n(s). 
 \]
Thus, at a new balls stage~$s$ we will be adding more than $n(s)$-many new balls to the root of the machine; and we will use the fact that $n(s)\to \infty$. As time goes by, we will be getting balls at the root in increasing number, and infinitely often, all of these balls will be $A$-correct.

\subsubsection{$\Delta_3$ guessing}  % (fold)

By assumption, $A$ is $\lowtwo$, i.e., $\Pi_2(A)\subseteq \Delta_3$. Thus, any Boolean combination of $\Pi_2(A)$ statements is equivalent to a $\Sigma_3$ statement. Since there is a universal $\Pi_2(A)$ set, this is effective: given~$\psi$, a Boolean combination of $\Pi_2(A)$ statements, we can effectively find a $\Sigma_3$ statement $\chi$ such that $\psi$ holds if and only if $\chi$ holds.

In turn, from $\chi$ we can effectively produce a uniformly computable collection of nondecreasing sequences $\seqell{\psi,n} = {(\oneell{{\psi,n}}{s})}$ for $n\in \Nat$ such that:
\begin{itemize}
	\item $\psi$ holds if and only if for some~$n$, $\seqell{\psi,n}$ is unbounded.
\end{itemize}
To see this, write $\chi$ as $\exists n\forall x\exists y\,\theta(n,x,y)$. Then we let $\oneell{\psi,n}{s}$ be the greatest $x<s$ such that for all $x'\le x$, there is some $y<s$ such that $\theta(n,x',y)$ holds. 

Shortly, for each node~$\alpha$ on the tree of strategies, we will formulate an ``$\alpha$ question'' $\psi(\alpha)$, which will be a $\Pi_2(A)$ statement. In other words, the collection of nodes~$\alpha$ such that $\psi(\alpha)$ holds is $\Pi_2(A)$. Using the transformation above for the statements $\psi(\alpha)$ and $\lnot\psi(\alpha)$ (both are Boolean combinations of $\Pi_2(A)$ statements), we obtain families of sequences $\seqell{\psi(\alpha),n}$ and $\seqell{\lnot\psi(\alpha),n}$. By the recursion theorem, we know a computable index for the construction, and so can have access to these families of sequences during the construction.\footnote{For the reader who may less comfortable with such a use of the recursion theorem, we provide a few details. For each~$e$, we can view the $e\tth$ partial computable function $\vphi_e$ as the function that tells us what the construction is: for each~$s$, we interpret $\vphi_e(s)$ as the code of the state of the construction at stage~$s$. For each~$e$ we can calculate a $\Pi_2(A)$ index for the set of~$\alpha$ for which $\psi(\alpha)$ holds according to the construction described by~$\vphi_e$, and using this, a computable index for all the resulting sequences $\seqell{\psi(\alpha),n}$ and $\seqell{\lnot\psi(\alpha),n}$; and given that, an index $f(e)$ such that $\vphi_{f(e)}$ describes the construction which is performed using these sequences. A fixed point $\vphi_e = \vphi_{f(e)}$ is the construction that we actually perform. Note that regardless if $\vphi_e$ is partial or not, $\vphi_{f(e)}$ is always total, and so the fixed point is total.}

\subsubsection{The $\alpha$ question}

\begin{definition} \label{def:maximal:alpha_question}
  Let $\alpha$ be a node. The statement $\psi(\alpha)$ is:
\begin{quote}
  For every $k\in \Nat$ there is an $A$-true stage~$s$ for which 
   \[
   | \Y{\alpha}{s}\cap W_{|\alpha|,s}|\ge k.
   \]
\end{quote}
\end{definition}

As just discussed, during the construction we have access to the sequences $\seqell{\psi(\alpha),n}$ and $\seqell{\lnot\psi(\alpha),n}$. Now for each node~$\beta$ we define a (single) sequence $\seqell{\beta}$ by induction on the length $|\beta|$: 
\begin{itemize}
  \item $\oneell{\emptystring}{s}= s$; 
  \item $\oneell{\beta\conc\infty_n}{s} = \min \{ \oneell{\beta}{s}, \oneell{\psi(\beta),n}{s}  \}$;
  \item $\oneell{\beta\conc\fin_n}{s} = \min \{ \oneell{\beta}{s}, \oneell{\lnot\psi(\beta),n}{s}  \}$.
\end{itemize}

Thus, via the sequences $\seqell{\gamma}$, the children~$\gamma$ of~$\beta$ together try to answer the~$\beta$ question. 

\begin{lemma} \label{lem:maximal:guessing_sequences_basics}
  For all~$\alpha$:
  \begin{sublemma}
  \item For all~$s$, $\oneell{\alpha}{s}\le s$; 
  \item If $\alpha\preceq \beta$ then $\oneell{\alpha}{s}\ge \oneell{\beta}{s}$. 
  \item If $\seqell{\alpha}$ is unbounded then there is some child~$\beta$ of~$\alpha$ such that $\seqell{\beta}$ is unbounded.     
  \end{sublemma}  
\end{lemma}

\begin{proof}
  Mostly immediate; (c) holds because for every~$\alpha$, one of $\psi(\alpha)$ and $\lnot\psi(\alpha)$ holds. 
\end{proof}

Also, by definition, $\seqell{\emptystring}$ is unbounded. We can therefore define:

\begin{definition} \label{def:maximal:true_path}
  The \emph{true path} is the path of nodes~$\alpha$ such that $\seqell{\alpha}$ is unbounded, but for every $\Left{\beta}{\alpha}$, $\seqell{\beta}$ is bounded. 
\end{definition}

In other words, $\emptystring$ is on the true path, and if $\alpha$ is on the true path, then the child of~$\alpha$ on the true path is the leftmost child~$\beta$ for which $\seqell{\beta}$ is unbounded. This is the leftmost child that guesses a correct witness for the $\Sigma_3$ version of $\psi(\alpha)$ or its negation. \Cref{lem:maximal:guessing_sequences_basics}(c) implies that the true path is infinite.

\subsubsection{Pulling and eliminating} 

We will now discuss the heart of the construction: how to decide where balls move, and which are enumerated into~$M$. The following terminology will be useful. 

\begin{itemize}
  \item A \emph{$\fin$-node} is a nonzero node~$\beta$ whose last entry is $\fin_n$ for some~$n$ (that is, $\beta =  (\beta^-)\conc \fin_n$ for some~$n$); and similarly, 
  \item An \emph{$\infty$-node} is a nonzero node whose last entry is $\infty_n$ for some~$n$. 
\end{itemize}

Here note that for a nonzero node~$\alpha$, we let $\alpha^-$ denote $\alpha$'s parent, the result of removing the last entry of~$\alpha$. 

\smallskip

The next definition will govern the movement and enumeration of balls. We will explain the details after we give the definition. The rough idea, though, is the following. If $\beta$ is an $\infty$-node of length $e+1$, then we will ensure that $\Y{\beta}\subseteq W_e$ by only allowing elements that have already entered~$W_e$ to enter $\Y{\beta}$. On the other hand, if $\beta$ is a $\fin$-node, then we get $\Y{\beta}\cap W_e =^* \emptyset$ ``for free'', by the fact that the $\beta^-$ question $\psi(\beta^-)$ fails. As discussed, in the first case, an $\infty$-node $\beta$ will want to enumerate into~$M$ any balls it sees that are outside $W_{e,s}$. If~$\beta$ is correct about its guess, then this will not cause $\Nat\subseteq^* M$, because $\psi(\beta^-)$ holds: sufficiently many balls will not be enumerated into~$M$. The complexity of the definitions comes from the need to deal with nodes to the left of the true path, that have wrong opinions. 

For any node~$\beta$, let 
 \[
 \Y{<_{L}\!\beta}{s} = \bigcup \,\,\{\Y{\gamma}{s}\,:\, \Left{\gamma}{\beta}\},
 \]
 and let 
  \[
  \Y{\le\! \beta}{s} = \Y{<_{L}\!\beta}{s}\,\cup \,\Y{\beta}{s} = \bigcup \,\,\{\Y{\gamma}{s}\,:\, {\gamma}\le {\beta}\}.
  \]
  Here recall that $\gamma\le \beta$ denotes the Kleene-Brouwer ordering, not the lexicographic ordering.

\begin{definition} \label{def:maximal:grabbability}
	A ball~$x$ is \emph{pullable} by a nonzero node~$\beta$ at stage~$s$ if:
	\begin{orderedlist}
	  	\item \label{item:maximal:grab:location} 
      $x\in \Y{\beta^-}{s}\setminus \Y{\le\!\beta}{s}$; 
	  	
      \item \label{item:maximal:grab:least} 
      $x > |\beta|$; 
      
      \item \label{item:maximal:grab:hungry} 
      $|\Y{\beta}{s}| < \oneell{\beta}{s}$; and
	  	
      \item \label{item:maximal:grab:W_e}  
      If $\beta$ is an $\infty$-node then $x\in W_{|\beta^-|,s}$. 
	  \end{orderedlist}  

  A ball~$x$ is \emph{eliminable} by a nonzero node~$\beta$ at stage~$s$ if:
  \begin{orderedlist}[resume]
      \item \label{item:maximal:eliminate:location} 
      $x\in \Y{\beta^-}{s}\setminus \Y{\le\! \beta}{s}$; 

      \item \label{item:maximal:eliminate:least} 
      $x \ne \min \Y{\beta^-}{s}$; and

      \item \label{item:maximal:eliminate:ell} 
      $x< \oneell{\beta}{s}$.      
      % \item $\beta$ is an $\infty$-node, and $x\notin W_{|\beta^-|,s}$. 
    \end{orderedlist}    
\end{definition}

Let us explain. \ref{item:maximal:grab:location} and \ref{item:maximal:eliminate:location}  mean that~$x$ resides at $\beta$'s parent $\beta^-$, or at or below one of $\beta$'s siblings to the right of~$\beta$. These are the regions from which~$\beta$ is allowed to pull.  Thus, balls on the machine will only move downward in the Kleene-Brouwer ordering: either to the left, or drop from a parent to a child. 

\ref{item:maximal:grab:least} will help ensure that each ball outside~$M$ eventually stops moving: it cannot keep being pulled by longer and longer nodes. \ref{item:maximal:eliminate:least} will help show that~$M$ is coinfinite. A node~$\beta^-$ on the true path will guard one ball (the smallest that it can see), and ensure that it is not enumerated into~$M$. 

The fact that an $\infty$-node $\beta$ is only allowed to pull balls in $W_e$ is in conflict with the requirement that $\Y{\beta}=^* \outside{M}$. It is simple to reconcile the conflict: if at a stage~$s$, $x$ is in a region from which~$\beta$ can pull, and $x\notin W_{e,s}$, then $\beta$ can simply enumerate~$x$ into~$M$ and remove it from the machine. The danger, of course, is that numbers go into $W_e$ slowly, and we may be too hasty in enumerating them into~$M$, risking $M=^*\Nat$. To avoid this, we use the sequences $\seqell{\gamma}$, that will ensure that each node $\gamma$ to the left of the true path eventually stops enumerating into~$M$ numbers from $\outside{A}$. If $\beta$ is an $\infty$-node on the true path, we need to ensure that $\Y{\beta}$ is infinite, i.e., that it sees enough numbers in~$W_e$ sufficiently early. We will show that this follows from $\psi(\beta^-)$ being true. 

Thus, \ref{item:maximal:grab:hungry} and \ref{item:maximal:eliminate:ell} ensure that nodes to the left of the true path eventually stop acting: they stop eliminating or pulling balls $x\in \outside{A}$. (Balls $x\in A$ are ``phantom'' balls; the right way to think about them is as if they were never there, and so pulling them does not count as acting.) This uses the fact that if $\gamma$ lies to the left of the true path, then $\seqell{\gamma}$ is bounded. Of course, these restrictions also restrain the action of a node~$\beta$ on the true path. Here, we note the difference between \ref{item:maximal:grab:hungry} and \ref{item:maximal:eliminate:ell}, i.e., the difference between the conditions $x<\oneell{\beta}{s}$ and $|\Y{\beta}{s}|<\oneell{\beta}{s}$. The issue is of timing. 

The main step of the verification will be to ensure that if~$\beta$ lies on the true path, then the number of balls in $\Y{\beta}$ during $A$-true stages is unbounded (this is \cref{lem:maximal:main_true_path_lemma} below). Suppose that this is known for $\beta^-$ (either by induction, or directly from $\psi(\beta^-)$, if $\beta$ is an $\infty$-node). We need to ensure that if the opportunity arises, in a true stage, then $\beta$ will act by pulling balls. It could be that at such stages, $\oneell{\beta}{s}$ is too small, compared to the size of the balls that need pulling. This is why we do not require that $x< \oneell{\beta}{s}$ when deciding whether to pull~$x$ or not; we need to consider the number of balls we currently have, i.e., $|\Y{\beta}{s}|$. 

On the other hand, there is no reason for haste in eliminating balls; we just need to ensure that in the limit, all but finitely many ``deviant'' balls do end up in~$M$ (these are the balls outside $W_e$, when~$\beta$ ensures that $\outside{M}\subseteq^* W_e$). So it is legitimate for us to require the stronger condition $x< \oneell{\beta}{s}$ when considering elimination. And in fact, it is important that we do so. The reason is a little delicate. Suppose that $\alpha$, the node on the true path, is a $\fin$-node, but that $\gamma$ is an $\infty$-child of $\alpha^-$ that lies to the left of~$\alpha$. Since $\gamma$ is allowed to only pull elements of $W_e$ (again $e = |\alpha^-|$), and~$W_e$ may be small (indeed, may be empty), it may be the case that $\Y{\gamma}$ never reaches the size $\lim_s \oneell{\gamma}{s}$. Thus, $\gamma$ is always ``hungry'' for balls. If we allowed $\gamma$ to eliminate balls whenever $|\Y{\gamma}{s}| < \oneell{\gamma}{s}$, it may never stop doing so. For pulling balls, the weaker condition $|\Y{\gamma}{s}| < \oneell{\gamma}{s}$ is sufficient, since if~$\gamma$ keeps pulling balls $x\notin \outside{A}$, we will argue that eventually the size of $\Y{\gamma}$ reaches $\lim_s \oneell{\gamma}{s}$, and $\gamma$ will stop acting. 

\smallskip

We will need the following for the construction:

\begin{lemma} \label{lem:maximal:leftmost_pullable_node}
  If $x$ is a ball on the machine at some stage~$s$, and is pullable by some node~$\beta$ at stage~$s$, then there is a $\le$-least such~$\beta$.  
\end{lemma}

\begin{proof}
  There are only finitely many nodes~$\alpha$ such that $x\in \Y{\alpha}{s}$, namely those $\alpha\preceq \gamma$ where $x\in \SSS{\gamma}{s}$. The ball~$x$ is pullable only by children of such nodes~$\alpha$, and ones which are $\le \gamma$. The order-type of all of these nodes (according to the Kleene-Brouwer ordering) is a well-ordering, namely of order-type~$\w$ (finitely many to the left of~$\gamma$, and the children of~$\gamma$). 
\end{proof}

\medskip

\subsection{Construction} % (fold)

At the beginning of stage~$s$, we already have specified $M_s$ and $\Y{\alpha}{s}$ % and $g^\alpha_s$ 
for all~$\alpha$. We start with $M_0 = \emptyset$, and no ball is on the machine at the beginning of stage~$0$. % and $g^\alpha_0(k)\diverge$ for all~$\alpha$ and~$k$.

At stage~$s$ there are three possibilities. 

\begin{sublemma}
	\item If there is a ball on the machine which is pullable by some node, then for each such ball~$x$, let $\beta$ be the $\le$-least node by which~$x$ is pullable; we move $x$  to~$\beta$ (by setting $x\in \SSS{\beta}{s+1}$).\footnote{Observe that in this case, we move \emph{all} balls on the machine that are pullable by some node, not only by the ``strongest'' or ``highest priority'' node, a notion that we did not define.} 

  \item If no ball on the machine is pullable by some node, but some ball on the machine is eliminable by some node, then we enumerate each such~$x$ into $M_{s+1}$ and remove it from the machine.

  \item If no ball on the machine is either pullable or eliminable, then $s$ is declared to be a new balls stage. Recall that this means that we set $n(s+1)=n(s)+1$, i.e., that possibly $A_{s+1}\ne A_s$.  
	\begin{orderedlist}
		\item We let $M_{s+1} = M_s\cup A_{s+1}$, and remove any $x\in A_{s+1}$ from the machine. 
		\item We place any $x\in Q_{s+1}\setminus M_{s+1}$ which is not already on the machine, at the root of the machine, i.e., we put it into $\SSS{\emptystring}{s+1}$.
	\end{orderedlist}
\end{sublemma}

\medskip

\subsection{Verification} % (fold)

\begin{lemma} \label{lem:maximal:Y_subset_of_Q}
	For all~$s$, $\Y{\emptystring}{s} =  Q_s\setminus M_s$. 
\end{lemma}

\begin{proof}
	By induction on~$s$. If $x\in \Y{\emptystring}{s+1}\setminus \Y{\emptystring}{s}$ then $x\in Q_{s+1}\setminus M_{s+1}$; if $x\in Q_s\setminus Q_{s+1}$ then $x\in A_{s+1}$ and so $x\in M_{s+1}$ and so $x\notin \Y{\emptystring}{s+1}$.  
\end{proof}

\subsubsection{Finite ball movement} % (fold)

As mentioned above, we say that~$x$ is a \emph{permanent resident} of a node~$\beta$ if $x\in \SSS{\beta}{s}$ for all but finitely many stages~$s$. We let~$\SSS{\beta}$ be the set of permanent residents of~$\beta$.

\begin{lemma} \label{lem:maximal:ball_movement_finite_part_1}
	If $x\in \bigcup_s \Y{\emptystring}{s}$ and $x\notin M$ then~$x$ is a permanent resident of some node. 
\end{lemma}

\begin{proof}  Since $x\notin M$, for all but finitely many~$s$, we have $x\in \Y{\emptystring}{s}$. By \ref{item:maximal:grab:least} of \cref{def:maximal:grabbability}, if $x$ resides at some node~$\beta$ at some stage, then $|\beta|\le x$. Among such nodes, $x$ only moves downward in the Kleene-Brouwer ordering. Since the subtree of nodes of length $\le x$ is well-founded (has no infinite path), the Kleene-Brouwer ordering restricted to such nodes is well-founded, and so, $x$ cannot move infinitely often. 
\end{proof}

\Cref{lem:maximal:ball_movement_finite_part_1} implies:

\begin{lemma} \label{lem:maximal:infinitely_many_new_balls}
	There are infinitely many new balls stages. 
\end{lemma}

Thus, as $n(s)\to \infty$, every $x\notin M$ is eventually put on the machine and is never removed from the machine. This shows:

\begin{lemma} \label{lem:maximal:ball_movement_finite_part_2}
	Every $x\notin M$ is a permanent resident of some node. 
\end{lemma}

We let $\Y{\alpha} = \lim_s \Y{\alpha}{s} = \bigcup_{\beta\succeq \alpha}\SSS{\beta}$; so $\Y{\emptystring} = \outside{M}$.

\subsubsection{True path} % (fold)

Recall that the true path (\cref{def:maximal:true_path}) is the collection of nodes~$\alpha$ such that $\seqell{\alpha}$ is unbounded, but for all $\Left{\gamma}{\alpha}$, $\seqell{\gamma}$ is bounded. We observed that the  true path is an infinite path of the tree. 

\begin{lemma} \label{lem:maximal:left_of_true_path}
	If~$\alpha$ lies on the true path, then there are only finitely many stages at which some node $\Left{\gamma}{\alpha}$ either pulls or eliminates a ball $x\in \outside{A}$. % or eliminates any ball. 
\end{lemma}

\begin{proof}
	This is proved by induction on the length of~$\alpha$; it is vacuously true for $\alpha = \emptystring$. Suppose that this is known for~$\alpha^-$. %Suppose that after stage~$s_0$, no node $\Left{\gamma}{\alpha^-}$ pulls or eliminates any ball $x\in \outside{A}$. %, and no such node eliminates any ball. 
  Let
	\[
		l = \max \set{\lim_{s\to \infty} \oneell{\beta}{s}}{\Left{\beta}{\alpha}\text{ is a child of~$\alpha^-$}}. 
	\]
	By \cref{lem:maximal:guessing_sequences_basics}(b),
	\[
		l = \max \set{\lim_{s\to \infty} \oneell{\gamma}{s}}{\Left{\gamma}{\alpha}\andd \gamma \succ \alpha^-}. 	
	\]
	Then by \ref{item:maximal:eliminate:ell} of \cref{def:maximal:grabbability}, no number $x\ge l$ is ever eliminated by any $\Left{\gamma}{\alpha}$ extending~$\alpha^-$ (indeed, whether it is an element of $\outside{A}$ or not). By induction (from left to right) on the children~$\beta$ of~$\alpha^-$ which lie to the left of~$\alpha$, we show that~$\beta$ pulls only finitely many $x\in \outside{A}$. Suppose that this is true for all $\Left{\gamma}{\beta}$. Then after some stage, every ball $x\notin A$ which is pulled by~$\beta$, is not later pulled by a node to the left of~$\beta$, and is not eliminated by any descendant of~$\beta$, and so is an element of~$\Y{\beta}$. So if~$\beta$ pulls infinitely many $x\in \outside{A}$, then~$\Y\beta$ is infinite. But then, $|\Y{\beta}{s}|\ge l$ for all but finitely many stages~$s$. By \ref{item:maximal:grab:hungry} of \cref{def:maximal:grabbability}, at such stages, $\beta$ pulls no balls, a contradiction. 
\end{proof}

As a result:

\begin{lemma} \label{lem:maximal:finite_stuff_on_left}
	If~$\alpha$ lies on the true path, then $\Y{<_{L}\!\alpha}$ is finite. 
\end{lemma}

The main lemma is the following.

\begin{lemma} \label{lem:maximal:main_true_path_lemma}
	If~$\alpha$ lies on the true path, then for every~$k$ there is some $A$-true stage~$s$ such that $|\Y{\alpha}{s}|\ge k$. 
\end{lemma}

\begin{proof}
	We prove this by induction on the length of~$\alpha$. First we consider $\alpha = \emptystring$. If~$s$ is a new balls stage and $s+1$ is $A$-true, then, as discussed above, at least~$n(s)$ many new balls from $Q_{s+1}$ are placed in $\SSS{\emptystring}{s+1}$; and $n(s)\to \infty$.  

  \smallskip

	Suppose that $\alpha\ne \emptystring$ lies on the true path, and that the lemma holds for~$\alpha^-$. Let~$s_0$ be a stage after which no $\Left{\gamma}{\alpha}$ pulls or eliminates any ball $x\in \outside{A}$; in particular, no $\Left{\gamma}{\alpha}$ pulls any ball during an $A$-true stage $s>s_0$. Further, if $s>s_0$ is $A$-true then $\Y{<_{L}\!\alpha}{s}= \Y{<_{L}\!\alpha}$, as all balls on the machine at stage~$s$ are from $\outside{A}$. Let $m = |\Y{<_{L}\!\alpha}|$, which by \cref{lem:maximal:finite_stuff_on_left} is finite. 

  Let $k\in \Nat$. There are two cases. 

  First, suppose that $\alpha$ is a $\fin$-node. By induction, there is an $A$-true stage $s>s_0$ for which 
  \[
    |\Y{\alpha^-}{s}| \ge  k + m + 1. 
  \]
  Choose such~$s$ sufficiently large so that $\oneell{\alpha}{s}>k$. If $|\Y{\alpha}{s}|\ge k$ then we are done. Otherwise, $|\Y{\alpha}{s}|<\oneell{\alpha}{s}$. Every $x\in \Y{\alpha^-}{s}\setminus \Y{\le\!\alpha}{s}$ larger than $|\alpha|$ is pullable by~$\alpha$ at stage~$s$. Now $\Y{\alpha^-}{s}$ contains at most one $x\le |\alpha|$, namely $x= |\alpha|$ (as every $x\in \Y{\alpha^-}{s}$ has size $> |\alpha^-|$)\footnote{For the purposes of this argument, let us agree that $0\notin \Nat$, so that this calculation applies to $\alpha^- = \emptystring$ as well.}; and as discussed, $|\Y{<_{L}\!\alpha}{s}|=m$. Hence, there remain at least $k-|\Y{\alpha}{s}|$ many balls in $\Y{\alpha^-}{s}$ that are pullable by~$\alpha$. Since $s>s_0$, and~$s$ is $A$-true, no node to the left of~$\alpha$ pulls such balls at stage~$s$. Hence, at stage~$s$, all balls pullable by~$\alpha$ will actually be pulled by~$\alpha$, and moved to $\Y{\alpha}{s+1}$. Also, no balls already in $\Y{\alpha}{s}$ will be pulled to the left. Since balls are pulled at stage~$s$, no balls are eliminated (anywhere on the machine) at stage~$s$ --- we only eliminate balls if no balls are pullable. So $\Y{\alpha}{s}\subseteq \Y{\alpha}{s+1}$, and overall, we see that $|\Y{\alpha}{s+1}|\ge k$. Since~$s$ is not a new balls stage, $s+1$ is also $A$-true, and so is as required. 

  If~$\alpha$ is an $\infty$-node, then the argument is the same, except that we need an $A$-true stage $s>s_0$ such that $\oneell{\alpha}{s}>k$ and such that 
  \[
    |\Y{\alpha^-}{s}\cap W_{|\alpha^-|,s}| \ge  k + m + 1,
  \]
  since $\alpha$ is allowed to pull only balls from $W_{|\alpha^-|,s}$. However, since $\seqell{\alpha}$ is unbounded, $\psi(\alpha^-)$ holds, which gives us exactly what we need. Note that in this case, we do not need to use the inductive hypothesis on $\alpha^-$, nor is it sufficient for our purposes.\footnote{Curiously, this means that we would not need to prove \cref{lem:maximal:main_true_path_lemma} by induction, if we already knew that there are infinitely many $\infty$-nodes on the true path. This is true, however, we need \cref{lem:maximal:main_true_path_lemma} to prove this.} 
\end{proof}

\begin{remark} \label{rmk:maximality:silliness_of_proof}
  During the proof of \cref{lem:maximal:main_true_path_lemma}, we relied on our stipluation that we do not eliminate any balls during a stage at which some balls are pullable. So we showed that we can get $|\Y{\alpha}{s+1}|\ge k$, but it is possible that immediately after that, some balls from $\Y{\alpha}{s+1}$ are eliminated, reducing the size of the set. It seems quite silly that \cref{lem:maximal:main_true_path_lemma} would depend on such an unimportant timing trick. 

  Indeed, it does not. We could allow balls to be pulled and eliminated at the same stage. But then, in order to prove \cref{lem:maximal:main_true_path_lemma}, we would need to first show that if more and more balls are supplied to $\Y{\alpha}$, then many of these balls would not be eliminated. We will now do this; it is merely a convenience for our presentation, to delay this part of the argument, instead of essentially proving \cref{lem:maximal:main_true_path_lemma} and \cref{lem:maximal:permanent_residents_on_true_path} together.
\end{remark}

\subsubsection{$M$ is not everything} % (fold)

\begin{lemma} \label{lem:maximal:permanent_residents_on_true_path}
	For every~$\alpha$ which lies on the true path, $\Y{\alpha} \ne \emptyset$. 
\end{lemma}

\begin{proof}
	Let~$\alpha$ be a node which lies on the true path; let $s_0$ be a stage after which no ball $x\notin A$ is pulled or eliminated by any node $\Left{\gamma}{\alpha}$. Let~$x$ be the smallest number in~$\outside{A}$ for which there are stages $t\ge s>s_0$ such that:
	\begin{itemize}[label=--]
		\item $x\in \Y{\alpha}{t}$; 
		\item $s$ is $A$-true and $x \le \max \Y{\alpha}{s}$. 
	\end{itemize}
	Such~$x$ exists by \cref{lem:maximal:main_true_path_lemma} (we can take $t=s$). Let stages $t \ge s > s_0$ witness~$x$. 

  By induction on stages $r\ge t$, we argue that $x\in \Y{\alpha}{r}$ and that $x = \min \Y{\alpha}{r}$. The minimality of~$x$ ensures that $x = \min \Y{\alpha}{t}$. Let $r\ge t$, and suppose that $x = \min \Y{\alpha}{r}$. By \cref{def:maximal:grabbability}\ref{item:maximal:eliminate:least}, at stage~$r$, $x$ is not eliminable by any child of~$\alpha$; since $r>s_0$, $x$ is not eliminated by any node at stage~$r$. Similarly, since $r>s_0$, $x$ is not pulled by any node to the left of~$\alpha$ at stage~$r$. Hence, $x\in \Y{\alpha}{r+1}$. Any $y<x$ on the machine at stage $r+1$ is an element of $\outside{A}$, as~$s$ was $A$-true, and so by minimality of~$x$, will not enter $\Y{\alpha}{r+1}$.  
\end{proof}

As a result: 

\begin{lemma} \label{lem:maximal_set_is_coinfinite}
	$M$ is coinfinite. 
\end{lemma}

\begin{proof}
	Since the true path is infinite, for every finite $D\subseteq \outside{M}$ there is some~$\alpha$ on the true path such that $D\cap  \Y{\alpha} = \emptyset$; $\Y{\alpha} \subseteq \outside{M}$  and so by \cref{lem:maximal:permanent_residents_on_true_path}, $\outside{M}\ne D$. 
\end{proof}

\subsubsection{Maximality} % (fold)

\begin{lemma} \label{lem:all_balls_on_true_path}
	For every~$\alpha$ on the true path, $\Y{\alpha} =^* \outside{M}$. 
\end{lemma}

\begin{proof}
	We know this holds for $\alpha = \emptystring$, and so it suffices to show that if $\alpha\ne \emptystring$ lies on the true path then $\Y{\alpha} =^* \Y{\alpha^-}$. By \cref{lem:maximal:finite_stuff_on_left}, $\Y{<_{L}\!\alpha}$ is finite, so it suffices to show that all but finitely many elements of $\Y{\alpha^-}\setminus \Y{<_{L}\!\alpha}$ are in $\Y{\alpha}$. Let $x\in \Y{\alpha^-}\setminus \Y{<_{L}\!\alpha}$, and suppose that $x>  \min \Y{\alpha^-}$. 

  If $s$ is sufficiently late, then $\oneell{\alpha}{s}>x$, $x\in \Y{\alpha^-}{s}\setminus \Y{<_{L}\!\alpha}{s}$, and $x>\min \Y{\alpha^-}{s}$. Then $x$ is eliminable by~$\alpha$ at~$s$. Hence, $x$ must in fact be pulled by~$\alpha$ at stage~$s$.%\footnote{Note that we allow multiple balls to be pulled by multiple nodes at the same stage; we do not need to wait for stronger nodes to act before~$\alpha$ can pull a ball.}
\end{proof}

\begin{remark} \label{rmk:maximality:more_silliness}
  Recall that the reason for eliminating balls is when we want $\outside{M}\subseteq^* W_e$; we then need to enumerate the complement of $W_e$ into~$M$. However, this was not incorporated into the definition of ``eliminable'', and it seems that \cref{lem:all_balls_on_true_path} relies on elements of $W_e$ being eliminable by $\fin$-nodes (when $|\Y{\alpha}{s}|\ge \oneell{\alpha}{s}$). This is not necessary. In \cref{def:maximal:grabbability}, we could add the clause ``$\beta$ is an $\infty$ node and $x\notin W_{|\alpha^-|,s}$'' to the definition of eliminability.\footnote{This would make it easier for us to allow pulling and eliminating balls at the same stage, if we so wished, as it would ensure than no ball is both pullable and eliminable by the same node; see \cref{rmk:maximality:silliness_of_proof}.} But then we would need to separate the proof of \cref{lem:all_balls_on_true_path} into cases. If $\alpha$ is an $\infty$-node, then the proof is as above. If it is a $\fin$-node, then we argue that as $\seqell{\alpha}$ is unbounded, it will eventually pull balls $x$ as above. 
\end{remark}

The proof of \cref{thm:Lachlan_maximal} is concluded with the following lemma.

\begin{lemma} \label{lem:maximality:maximality}
	$M$ is maximal. 
\end{lemma}

\begin{proof}
	Let $e\in \Nat$; let $\alpha$ be the node on the true path of length $e+1$. 

	\smallskip
	
	First, suppose that $\alpha$ is an $\infty$-node. Every ball pulled by~$\alpha$ is already an element of~$W_e$, so $\Y{\alpha}\subseteq W_e$. By \cref{lem:all_balls_on_true_path}, $\outside{M}\subseteq^* W_e$. 

	\smallskip
	
	Next, suppose that~$\alpha$ is a $\fin$-node. Since $\seqell{\alpha}$ is unbounded, we know that $\psi(\alpha^-)$ fails: there is some~$k$ such that for every $A$-true stage~$s$, 
	\[
		|\Y{\alpha^-}{s}\cap W_{e,s} | \le k. 
	\]
	Then $|\Y{\alpha^-}\cap W_e| \le k$. For otherwise, there would be a set $D\subseteq \Y{\alpha^-}\cap W_e$ of size $k+1$. But then, for all but finitely many stages~$s$, $D\subseteq \Y{\alpha^-}{s}\cap W_{e,s}$, so $|\Y{\alpha^-}{s}\cap W_{e,s} | \ge k+1$ for some $A$-true stage~$s$, which is not the case. It follows that $\outside{M}\cap W_e =^*\emptyset$. 
\end{proof}

\section{Atomless supersets} 
\label{sec:atomless_supersets}

We modify the construction above to prove \cref{thm:atomless}: every coinfinite $\lowtwo$ c.e.\ set has an atomless, hyperhypersimple superset. 

The new ingredient is a process for splitting a stream of balls in two, each infinite outside of~$A$.

\subsubsection{Boolean algebras and binary trees}

Recall that we can generate Boolean algebras from trees. If $T\subseteq 2^{<\w}$ is a tree, then $B(T)$ is the quotient of the free Boolean algebra with generators $\s\in T$, modulo the relations:
\begin{itemize}
  \item Every $\tau\in T$ is the join of its children;
  \item If $\s,\tau\in T$ are incomparable then $\s\wedge \tau = 0_{B(T)}$. 
\end{itemize}
Equivalently, we can think of $B(T)$ as the collection of finite unions of the sets $[\s]\cap [T]$, for all $\s\in T$, ordered by set inclusion; in this version, $\tau$ is identified with $[\tau]\cap [T]$. 

Note that this implies that if $\tau$ is a leaf (or more generally, if $\tau$ is not extendible to an infinite path on~$T$), then $\tau = 0_{B(T)}$; and that $\emptystring=1_{B(T)}$. Every countable Boolean algebra can be presented in such a way (up to isomorphism) by considering a tree of nonzero finite Boolean combinations of any sequence of generators.\footnote{In greater detail: if $\{b_n\}$ is a set of generators of~$B$, then we let $T$ be the set of all sequences~$\s$ such that $b_\s = \bigwedge \left\{ b_i \,:\,  \s(i)=1 \right\} \wedge \bigwedge \left\{ \lnot b_i \,:\,  \s(i)=0 \right\}$ is nonzero. See the proof of \cite[Thm.\:X.7.2]{Soa}} For our purposes this section, we rely on the fact that the atomless Boolean algebra is $B(2^{<\w})$, the one generated by the full binary tree.\footnote{In general, the atoms of $B(T)$ are those nodes that isolate a path, so if $T\subseteq 2^{<\w}$ has no leaves, then $B(T)$ is atomless if and only if it is perfect. Among all perfect trees, we choose $T = 2^{<\w}$ as it makes notation simplest.} Thus, given a coinfinite $\lowtwo$ c.e.\ set~$A$, we will enumerate a superset $H\supseteq A$, and construct a tree of sets $\smallseq{Z(\rho)\,:\, \rho\in 2^{<\w}}$ satisfying:
\begin{orderedlist}
  \item $Z(\emptystring)=^* \outside{H}$; 
  \item For each $\rho$, $Z(\rho)$ is infinite;
  \item For each $\rho$, $H\cup Z(\rho)$ is c.e.;
  \item For all $\rho\in 2^{<\w}$, $Z(\rho\conc 0)$ and $Z(\rho\conc 1)$ are disjoint, and $Z(\rho) =^* Z(\rho\conc 0)\cup Z(\rho\conc 1)$;
  \item For every c.e.\ set $W\supseteq H$ there is a finite set $D\subset 2^{<\w}$ such that $W\setminus H=^*\bigcup \left\{ Z(\rho) \,:\,  \rho\in D \right\}$. 
\end{orderedlist}
This suffices: to see that $\+L^*(H)$ is a Boolean algebra, it suffices to show that it is complemented (see \cite[Sec.\:X.2]{Soa}). Indeed, if $W\supseteq H$ is c.e., by~(v), let $D\subseteq 2^{<\w}$ be a finite set such that $W\setminus H =^* \bigcup \left\{ Z(\rho) \,:\,  \rho\in D \right\}$. By~(iv), we may assume that all elements of~$D$ have the same length, say~$k$. Let $E = \{0,1\}^k\setminus D$; then $H\cup \bigcup \left\{ Z(\rho) \,:\,  \rho\in E \right\}$ is a complement of~$W$ in $\+L^*(H)$; by (iii), it is c.e. (ii) and (iv) also ensures that $\+L^*(H)$ has no atoms. 

As we are using the $\Delta_3$ machinery, the sets $Z(\rho)\cup H$ will not be uniformly c.e.; rather, they will be $=^*$ to sets whose indices we can read off the true path of the construction.

\subsection{The setup} % (fold)

We will use the same general mechanism as above. We will have an infinite-branching tree of strategies, that will serve as a pinball machine on which we move balls around. When numbers are enumerated into~$H$, they will be removed from the machine; the complement $\outside{H}$ will consist of the balls that have permanent residence at some node on the tree. 

Toward enumerating the sets $Z(\rho)$, each ball $x$ on the machine at some stage~$s$ will have a \emph{label} $\rho\in 2^{<\w}$. As with residence on the tree, the label of a ball $x\notin H$ will stabilise (indeed, the label can change only when~$x$ moves). If the final label of~$x$ is~$\rho$, then we will put~$x$ into $Z(\rho')$ for all $\rho'\preceq \rho$. 

We will use the notation $\SSS{\alpha}{s}$, $\Y{\alpha}{s}$, $\Y{<_{L}\!\alpha}{s}$ etc.\ as above. We will use similar notation that also specifies labels. Namely, for a node~$\alpha$ on the tree and $\rho\in 2^{<\w}$, we will let
$\SSS{\alpha,\rho}{s}$
denote the collection of all balls $x$ which at the beginning of stage~$s$, reside at~$\alpha$ and have label~$\rho$. We will let
\[
  \Y{\alpha,\rho}{s} = \bigcup \left\{ \SSS{\alpha',\rho'}{s}\,:\,  \alpha\preceq \alpha' \andd \rho\preceq \rho' \right\}.
\]
That is, $\Y{\alpha,\rho}{s}$ consists of the balls~$x$ which at the beginning of stage~$s$, 
\begin{itemize}
  \item \emph{reside} at some $\alpha'\succeq \alpha$, and
  \item have a \emph{label} $\rho'\succeq \rho$. 
\end{itemize}
So as discussed, the elements of $\Y{\alpha,\rho}{s}$ are those elements in $\Y{\alpha}{s}$ that at stage~$s$ we intend to put in $Z(\rho)$. So at the end of the verification, we will let $Z(\rho) = \Y{\emptystring,\rho}$, the collection of all balls $x\in \Y{\emptystring}$ whose permanent label extends~$\rho$. As in the previous construction, if $\alpha$ is on the true path, then $\Y{\alpha} =^* \outside{H}$, so $Z(\rho)=^* \Y{\alpha,\rho}$.

\subsubsection{The requirements}

In the current construction, we will have two kinds of requirements. The analogues of the maximality requirements from the previous construction are the requirements for meeting (v) above. For each $\rho\in 2^{<\w}$ of length~$e$, we will ensure that either $Z(\rho)\subseteq^* W_e$ or $Z(\rho)\cap W_e=^*\emptyset$. Thus, $D = \left\{ \rho\in \{0,1\}^e \,:\,  Z(\rho)\subseteq^* W_e \right\}$ will show that~(v) holds for~$W_e$.  The way to meet these requirements will be very similar to the previous construction, except that the $\alpha$ question will be more complicated; for each $\rho\in \{0,1\}^e$, we will need to guess whether we will see enough balls to make $Z(\rho)\subseteq W_e$ while ensuring that $Z(\rho)$ is infinite. The mechanism of pulling and eliminating will be used to achieve that, but we will see that the definitions of ``pullable'' and ``eliminable'' need to be more complicated.

The other requirements are new: they ensure that (iv) above holds. Namely, for each~$\rho$, we need to ensure that $Z(\rho\conc 0)$ and $Z(\rho\conc 1)$ form a splitting of $Z(\rho)$ (up to finite differences), into two infinite sets. In the construction, we will have nodes devoted to such a requirement. The task for such a node is to take balls with label~$\rho$, and decide whether to extend their label to either $\rho\conc 0$ or $\rho\conc 1$. The difficulty, of course, is that during the construction, we do not know whether a ball~$x$ is in~$A$ or not. It would be very bad if all but finitely many balls $x$ that we direct to $\rho\conc 0$, for example, will end up in~$A$, as that would make $Z(\rho\conc 0)$ finite. Unlike the other kind of requirement, $\Delta_3$ guessing is not sufficient for this task: we know that there will be infinitely many $x\notin A$ with label~$\rho$. We need to somehow obtain individual balls that we have a good reason to guess are not in~$A$. 

Our main contribution is precisely this: a new method for \emph{certifying} that groups of balls are not in~$A$. As we will explain later, in fact, to perform this certifying and splitting, it will be notationally convenient to have not just one level of nodes devoted for each requirement, but two; nodes and their children together will perform these tasks. 
 
So we will have two kinds of nodes. 
\begin{orderedlist}
  \item \emph{Decision nodes}, of length $3e$, whose task, for each $\rho\in \{0,1\}^e$, to decide~$W_e$ on $\Y{\alpha,\rho}$. We call a node of length $3e$ an \emph{$e$-decision node}.  
  
  \item \emph{Splitting nodes} and their \emph{children}, of lengths $3e+1$ and $3e+2$. The children of a decision node~$\alpha$ of length $3e+1$ will pull balls with labels~$\rho$ of length~$e$, and decide to extend their label to either $\rho\conc 0$ or $\rho\conc 1$. We call a node of length $3e+1$ a \emph{parent $e$-splitting node}, and a node of length $3e+2$ a \emph{child $e$-splitting node}.   
\end{orderedlist}

\subsubsection{True balls rather than true stages} 

As above, we will make use of a true-stage enumeration $(\tilde A_r)$ of~$A$ with respect to the function $f(r)=r^2$. We will use the same mechanism as above to slow down this enumeration to an enumeration $(A_s)$ defined during the construction, again letting $n(s)=r$ when $A_s = \tilde A_r$. As in the previous construction, if we move or enumerate balls, or perform any other action, at a stage~$s$, then  will set $A_{s+1}=A_s$ (by setting $n(s+1)=n(s)$). When no other action is taken at stage~$s$, we will declare~$s$ to be a new balls stage, and set $n(s+1)=n(s)+1$. 

As in the very first part of the proof of \cref{lem:maximal:main_true_path_lemma}, this ensures that infinitely often, the \emph{root} of the tree receives large collections of balls which are all outside~$A$. 

However, unlike the previous construction, there will be an element of delay to ball movement. In order to ensure that both $\Y{\alpha,\rho\conc 0}$ and $\Y{\alpha,\rho\conc 1}$ receive many balls outside~$A$, a child of a splitting node~$\alpha$ will hold on to many balls in $\Y{\alpha,\rho}$ until it receives confirmation from $\emptyset'$ that these balls are indeed outside~$A$. (Infinitely often, this confirmation will be incorrect, but infinitely often it will be correct, and this will be sufficient for our purposes; we will see that what is important, is that each \emph{individual attempt} at certification makes only finitely many mistakes.) The stages at which such confirmation is received do not need to line up with the $A$-true stages. So we will not be able to directly prove \cref{lem:maximal:main_true_path_lemma} for the current construction. Rather, instead of looking for ``completely true'' \emph{stages}, at which every ball on the machine is correct, we will simply ensure that we get more and more true \emph{balls}.

\begin{definition} \label{def:true_balls}
  A number $x\in \outside{A}_s$ is \emph{$A$-true} at stage~$s$ if 
\[
   A\rest{x+1} = A_s\rest{x+1}.
 \] 
\end{definition}
That is, not only $x$ will not enter~$A$ in the future, but no number $y\le x$ currently outside~$A$ will enter~$A$ in the future. If we get enough of these, we do not really care that larger balls are not $A$-true at the same stage. 

\begin{definition} \label{def:Q_and_C}
  As in the previous construction, we let~$Q_s$ be the set of the $n(s)^2$-smallest elements of $\outside{A}_s$. We let~$C_s$ denote the collection of $x\in Q_s$ which are $A$-true at stage~$s$. 
\end{definition}

\subsubsection{The questions and the tree} % (fold)

A decision node~$\alpha$ needs to decide, for each $\rho$ of length~$e$, whether to make $Z(\rho)$ almost contained in~$W_e$ or almost disjoint from~$W_e$. This means that the $\alpha$ question is more complicated. Let~$\alpha$ be an $e$-decision node. For each set $D\subseteq \{0,1\}^e$, the statement $\psi(\alpha,D)$ says:
\begin{quote}
  $D$ is the set of $\rho\in \{0,1\}^{e}$ for which for every~$k$ there is some~$s$ such that 
  \[
    |\Y{\alpha,\rho}{s}\cap W_{e,s} \cap C_s | \ge k. 
  \]
\end{quote}

The outcomes of~$\alpha$ are 
\[
	D_n
\]
for each $D\subseteq \{0,1\}^{e}$ and each $n\in \Nat$, ordered in order-type~$\w$ in some way. Note that $\psi(\alpha,D)$ will hold for precisely one $D\subseteq \{0,1\}^e$, and we will ensure that $Z(\rho)\subseteq^* W_e$ if and only if $\rho\in D$ (and otherwise, $Z(\rho)\cap W_e =^*\emptyset$). 

Since each $\psi(\alpha,D)$ is a finite Boolean combination of $\Pi_2(A)$ statements, it is equivalent to a $\Sigma_3$ statement, so as in the previous construction, we obtain uniformly computable, nondecreasing sequences $\seqell{\psi(\alpha,D),n}$, such that $\psi(\alpha,D)$ holds if and only if for some~$n$, $\seqell{\psi(\alpha,D),n}$ is unbounded. As above, we use these to define sequences $\seqell{\alpha}$ for nodes~$\alpha$; these will only be used for nodes~$\alpha$ that are the \emph{children} of decision nodes. So we recursively define:
\begin{itemize}
  \item If $\alpha = (\alpha^-)\conc D_n$ is the child of a decision node~$\alpha^-$, then we let $\oneell{\alpha}{s}$ be the minimum of $\oneell{\psi(\alpha^-,D),n}{s}$, and $\oneell{\beta}{s}$ for any $\beta\precneq \alpha$ that is also the child of a decision node. 
\end{itemize}

\smallskip

It is more difficult to explain now why we need both parent and children splitting nodes. If $\alpha$ is a parent $e$-splitting node, then its children are $\alpha\conc n$ for all $n\in \Nat$, ordered naturally. If $\beta$ is such a child, then~$\beta$ has a unique child $\beta^+$ on the tree (which will in turn be an $e+1$-decision node). The outcomes~$n$ of~$\alpha$ do not quite represent guesses about the behaviour of~$\alpha$; we will discuss them later.

\subsubsection{Updating previous definitions} % (fold)

We update notation from the previous construction. For a node~$\alpha$ and $\rho\in 2^{<\w}$, and a stage~$s$, we let:
\begin{itemize}
  \item $\Y{<_{L}\!\alpha,\rho}{s}$ be the collection of balls that at the beginning of stage~$s$, reside at some node $\beta$ that lies to the left of~$\alpha$, and have label extending~$\rho$. That is, $\Y{<_{L}\!\alpha,\rho}{s} = \bigcup_{\Left{\beta}{\alpha}} \Y{\beta,\rho}{s} =
  \Y{<_{L}\!\alpha}{s}\cap \Y{\emptystring,\rho}{s}$; 
  \item $\Y{\le\!\alpha,\rho}{s} = \Y{<_{L}\!\alpha,\rho}{s}\cup \Y{\alpha,\rho}{s}$. 
\end{itemize}

\begin{definition} \label{def:atomless:decision_pullable}
	Let $\alpha  = (\alpha^-)\conc D_n$ be a child of an $e$-decision node, and let $\rho\in \{0,1\}^{e}$.
 
 A ball~$x$ is \emph{pullable} by $(\alpha,\rho)$ at stage~$s$ if the following all hold:
	\begin{orderedlist}
		\item $x\in \Y{\alpha^-,\rho}{s}\setminus \Y{\le\! \alpha,\rho}{s}$;
		\item  \label{item:atomless:pull:size}
      $x > |\alpha|$; 
    \item \label{item:atomless:pull:new}
    either:
    \begin{itemize}
      \item $|\Y{\alpha,\rho}{s}| < \oneell{\alpha}{s}$, or
      \item $\Y{\alpha,\rho}{s}\ne \emptyset$ and $x< \max \Y{\alpha,\rho}{s}$;
    \end{itemize}
		\item if $\rho\in D$ then $x\in W_{e,s}$. 
	\end{orderedlist}

 A ball~$x$ is \emph{eliminable} by $(\alpha,\rho)$ at stage~$s$ if:
	\begin{orderedlist}[resume]
		\item $x\in \Y{\alpha^-,\rho}{s}\setminus \Y{\le\!\alpha,\rho}{s}$;
		\item $x\ne \min \Y{\alpha^-,\rho}{s}$; and
		\item $x< \oneell{\alpha}{s}$. 
	\end{orderedlist}
\end{definition}

We need to comment on the new part of this definition, namely the second part of \ref{item:atomless:pull:new}. The issue is that, as mentioned above, we will not be able to directly show that $\Y{\alpha}{s}$ is ``full'' during $A$-true stages (in the sense of \cref{lem:maximal:main_true_path_lemma}), only that it gets ``filled'' by $A$-true balls (see \cref{prop:atomless:main_true_path_lemma}(b) below). Consider a node~$\alpha$ on the true path (a child of a decision node). At a stage~$s$, the parent $\alpha^-$ may have many $A$-true balls. At the same stage, $\alpha$ already has larger balls and so does not feel ``hungry'' ($|\Y{\alpha,\rho}{s}| \ge \oneell{\alpha}{s}$). But because~$s$ is not an $A$-true stage, the balls that $\alpha$ has may be false. The problem is solved by allowing~$\alpha$ to pull balls smaller than ones it already has, even if it does not feel ``hungry''. This modification is sufficiently tame so that nodes on the left of the true path will again only pull finitely many balls from $\outside{A}$.\footnote{It may be difficult to envision how this situation can come about, since smaller balls are put on the machine earlier than larger balls. However, the collection of balls pullable by a node is not always an initial segment of the collection of available balls, since only some may be elements of some~$W_e$, whereas the others may not yet be eliminable by the same node. So the order ``breaks'' as balls move on the machine.}

\subsubsection{Certification} % (fold)

Since~$A$ is $\lowtwo$ and c.e., some $\emptyset'$-computable function dominates all $A$-computable functions (this follows from Martin's characterisation~\cite{mart} of the high degrees as those that compute functions that dominate all computable functions). 

\begin{definition} \label{def:dominating_function}
  We fix $\domin$ to be a $\emptyset'$-computable function that dominates all $A$-computable functions. We also fix a computable approximation $(\domin_s) = (\domin_0,\domin_1,\dots)$ of~$\domin$.
  %  with the following properties:
  % \begin{itemize}
  %   \item The approximation is nondecreasing in both variables: for all~$n$ and~$s$, $\domin_s(n)\le \domin_s(n+1)$ and $\domin_s(n)\le \domin_{s+1}(n)$. 
  %   \item for all~$n$ and~$s$, $\domin_s(n)\le s$. 
  % \end{itemize}
\end{definition}

Let $\alpha$ be a parent $e$-splitting node, and fix $\rho\in \{0,1\}^e$. The very rough idea of splitting is that $\alpha$ will define \emph{blocks} of balls, and hold on to these balls until $\domin$ gives confirmation to release these balls to nodes below. The $k\tth$ block will contain at least $2k$ many balls. For now, let us forget about $\alpha$'s children, and imagine that released balls are passed to the next splitting node (if $\alpha$ did not have multiple children then there would be a unique $(e+1)$-splitting node extending~$\alpha$). When the $k\tth$ block of balls is \emph{certified} by $\domin$, then~$k$ of the balls in that block will receive the label $\rho\conc 0$, and the rest, the label $\rho\conc 1$. 

For this purpose, we will define an $A$-computable function $f^{\alpha,\rho}$. An input $k\in \Nat$ for $f^{\alpha,\rho}$ indicates an attempt to capture and test the $k\tth$ block of balls. Suppose that at some stage~$r$, for all $m<k$, the $m\tth$-block of balls is currently defined, and that~$\alpha$ holds $2k$ many balls with label~$\rho$ (that are not associated with any existing block). We would then declare these balls to constitute the $k\tth$ block, and define $f^{\alpha,\rho}_{r+1}(k)> \domin_{r}(k)$. The $A_{r+1}$-\emph{use} $u = u^{\alpha,\rho}_{r+1}(k)$ of this computation will bound at least $2k$-many elements of the block. 

We then wait for a stage $s>r$ at which one of two things happens. Either $A_s\rest{u}\ne A_{r+1}\rest{u}$, in which case $f^{\alpha,\rho}_s(k)\diverge$. This $A$-change likely involves some elements of the block entering~$A$. We will then wait for~$\alpha$ to obtain new balls, that will allow us to redefine a new $k\tth$ block. 

Otherwise, we hope to see $\domin_s(k)> f^{\alpha,\rho}_s(k) = f^{\alpha,\rho}_{r+1}(k)$. This means that $\domin_s(k)> \domin_r(k)$. We will regard this as evidence that elements from the block will not later enter~$A$, and say that the block is \emph{certified}. We then process the block as described above, splitting it evenly between $\rho\conc 0$ and $\rho\conc 1$, and passing the balls in the block to the next node below. 

Of course, it is possible that neither of these events happen: the computation $f^{\alpha,\rho}_{r+1}(k)$ is never undefined, but $\domin(k)$ never exceeds the value of that computation. Therefore, while we wait, we will try to define the $(k+1)\tth$-block, and so on. 

The argument that this works will be by contradiction. We want to argue that infinitely many blocks will be released, and infinitely many of those will be $A$-correct, i.e., only contain balls from~$\outside{A}$. Assuming this is not the case, we will want to show that $f^{\alpha,\rho}$ is a total function. If this is shown, then we obtain the desired contradiction from the fact that~$\domin$ dominates $f^{\alpha,\rho}$, showing that for almost all~$k$, the correct $k\tth$ block will in fact be certified and released. To show that $f^{\alpha,\rho}$ is total (under the assumption for contradiction), we use the fact that $\vphi_s(k)$ changes only finitely many times. This implies that the $k\tth$ block will be released only finitely many times, and so eventually, an $A$-correct $k\tth$ block will be defined, showing that $f^{\alpha,\rho}(k)\converge$. 

Thus, the use of the certification process is in the delay it imposes on the release of balls. If we didn't wait for certification, then it is possible that we would keep redefining and releasing some $k\tth$ block, none of whose versions is $A$-correct. It would then be possible that of the blocks that we release, all but finitely many of the balls that we target to $\rho\conc 0$, say, end up in~$A$, resulting in $Z(\rho\conc 1)=^*Z(\rho)$, so the splitting requirement is not met. 

\smallskip

We now come to the need for multiple children. The reason is delicate. We want to argue that an $A$-correct definition of some $f^{\alpha,\rho}(k)$ will eventually be made. Imagine the following sequence of events:
\begin{enumerate}
  \item $f^{\alpha,\rho}(k)$ is defined at some stage. 
  \item While waiting for certification, we also define $f^{\alpha,\rho}(k+1)$ (after all, we do not know where $\domin$ starts dominating). 
  \item The $(k+1)\tth$ block is certified and released, but not the $k\tth$ one. 
  \item Some $x$ in the $k\tth$ block enters~$A$, making both $f^{\alpha,\rho}(k)$ and $f^{\alpha,\rho}(k+1)$ undefined. 
\end{enumerate}
Now, all we know, by induction, is that an analogue of \cref{lem:maximal:main_true_path_lemma} holds for~$\alpha$: for all $m$, there is some stage~$s$ at which $C_s\cap \Y{\alpha,\rho}{s}$ has size at least $m$. If~$s$ is such a stage after these events unfolded, then it is possible that the bulk of $C_s\cap \Y{\alpha,\rho}{s}$ is actually lying well below~$\alpha$ --- it is part of the $(k+1)$-block that was released at step~(3) above. Since these balls are not currently residing at~$\alpha$, we cannot use them to define a new $k\tth$-block. Unless\dots we pull them back up to~$\alpha$ from wherever they are currently residing. 

Now, in all similar constructions, pulling balls from below is a source of innumerable problems. Just as a simple example, it becomes much more difficult to argue that balls outside~$H$ will eventually reach a permanent residence, since pulling balls back up is \emph{increasing} in the Kleene-Brouwer ordering. More seriously, it will be difficult to show that the sets $H\cup Y(\gamma,\rho)$, for~$\gamma$ on the true path, are c.e. So we will not do this. A work-around is to allow~$\alpha$ to have infinitely many children $\alpha\conc n$. While waiting for certification, each block will ``reside'' at one of the children. Indeed the block waiting at a child~$\beta$ will simply be $\SSS{\beta,\rho}{s}$. 

While the $k\tth$ block is waiting at $\alpha\conc n$, all blocks defined later will reside at children of~$\alpha$ that lie to the right of $\alpha\conc n$, namely, children $\alpha\conc m$ for various $m>n$. When the $k\tth$ block is dissolved (at step~(4) above), it is now fine for $\alpha\conc n$ to pull all balls from the right, even those that were released to lower nodes, and use them to constitute a new $k\tth$ block.  

One can think of $\alpha\conc n$ as representing the guess that $\domin$ starts majorising $f^{\alpha,\rho}$ from input~$n$. This is not exactly right but the intuition is not far off. 

Another point to make is that for distinct~$\rho$, the function $\domin$ may start dominating the various functions $f^{\alpha,\rho}$ at different locations. The argument above will show that for each $\rho\in \{0,1\}^e$, there is some child~$\beta$ that releases $\rho$-blocks infinitely often. However, to obtain a child of~$\alpha$ on the true path, we need the \emph{same~$\beta$} to work for all~$\rho$. Thus, in the definition of certification below, we will let a child~$\beta$ hold on to its blocks, until all of them (one for each~$\rho$) are certified, and only then will it release them all. 

\smallskip

We can now give the details. Let $\alpha$ be a parent $e$-splitting node; let $\rho\in \{0,1\}^e$. As discussed, We will define a function $f^{\alpha,\rho}$, with intended oracle~$A$. So at various stages~$s$, for various~$k$, we may have $f^{\alpha,\rho}_s(k)\converge$ or not, and if it is defined, then we will declare a use $u = u^{\alpha,\rho}_s(k)$ for this computation. The usual rule applies: if $f^{\alpha,\rho}_s(k)\converge$, with use~$u$, and $A_s\rest{u} =A_{s+1}\rest{u}$, then $f^{\alpha,\rho}_{s+1}(k)\converge = f^{\alpha,\rho}_s(k)$ with the same use. If, on the other hand, $A_s\rest{u}  \ne A_{s+1}\rest{u}$, then we declare that $f^{\alpha,\rho}_{s+1}(k)\diverge$. Note that this will only happen if~$s$ is a new balls stage, and during such stages, we do not define new computations $f^{\alpha,\rho}_{s+1}(k)$. 

We will define $f^{\alpha,\rho}_s(k)$ only for $k\ge 1$, since there is no point in dealing with a block of size~0.\footnote{And it is notationally simpler than requiring the $k\tth$ block to have size $2k+2$.}

We will ensure that if $k\ge 1$ and $f^{\alpha,\rho}_s(k+1)\converge$ then $f^{\alpha,\rho}_s(k)\converge$ as well, that is, the domain of $f^{\alpha,\rho}_s$ is a (finite) initial segment of $\Nat\setminus\{0\}$. See \cref{lem:atomless:block_location} below. 

When we define a computation $f^{\alpha,\rho}_r(k)$ (at some stage $r-1$), then we will declare that the $k\tth$ block is currently \emph{waiting} at some child~$\beta$ of~$\alpha$; we will record this by setting $k^{\rho}_s(\beta) = k$ (note that $\alpha$ is determined by~$\beta$). This notation implies that at most one $\rho$-block is waiting at $\beta$ at a given stage. 

If  $k^{\rho}_s(\beta)\converge=k$ then $f^{\alpha,\rho}_s(k)\converge$. Thus, if $k^{\rho}_s(\beta)\converge=k$ but $f^{\alpha,\rho}_{s+1}(k)\diverge$ then we declare that $k^{\rho}_{s+1}(\beta)\diverge$. That is, if the $k\tth$ block is waiting at~$\beta$ at (the beginning of) stage~$s$, but is dissolved when passing from stage~$s$ to stage $s+1$, then no block is waiting at~$\beta$ at stage $s+1$. 

If $\beta$ \emph{releases} balls at stage~$s$, then we will also set $k^{\alpha,\rho}_{s+1}(\beta)\diverge$, even though $f^{\alpha,\rho}_{s+1}(k)\converge$; at stage $s+1$, the block does not wait at~$\beta$  anymore.   

We will ensure that if $k^{\rho}_s(\beta)\converge$, then for every child $\gamma$ of~$\alpha$ that lies to the left of~$\beta$, we have $k^{\rho}_s(\gamma)\converge$ as well (and $k^{\rho}_s(\gamma)<k^{\rho}_s(\beta)$). Again, see \cref{lem:atomless:block_location}.

\begin{definition} \label{def:atomless:certified}
	Let~$\alpha$ be a parent $e$-splitting node, let $\rho\in \{0,1\}^{e}$, and let~$\beta$ be a child of~$\alpha$.  Let~$s$ be a stage. 

  \begin{sublemma}
  \item A ball $x$ is \emph{pullable} by $(\beta,\rho)$ at stage~$s$ if:
    \begin{orderedlist}
      \item $x\in \Y{\alpha,\rho}{s}\setminus \Y{\le\!\beta,\rho}{s}$; and
      
      \item \label{item:atomless:def_certified:pullable:nothing_waiting} 
      $k^\rho_s(\beta)\diverge$ (that is, no $\rho$-block is waiting at~$\beta$ at stage~$s$). 
      % either:       \label{item:atomless:def_certified:pullable:two_cases} 
      % \begin{itemize}
      %   \item no $(\alpha,\rho)$-interval (an interval $I^{\alpha,\rho}_s(k)$ for some~$k$) is waiting at~$\beta$ at stage~$s$; or
      %   \item an interval $I^{\alpha,\rho}_s(k)$ is waiting at~$\beta$ at stage~$s$, and $x\in I^{\alpha,\rho}_s(k)$. 
      % \end{itemize}
    \end{orderedlist}

    \item Let $k\ge 1$. We say that $(\beta,\rho,k)$ is \emph{ready for definition} at stage~$s$ if:
    \begin{orderedlist}[resume]
      \item \label{item:atomless:def_certified:ready:k-1}
      $f^{\alpha,\rho}_s(k)\diverge$, and $f^{\alpha,\rho}_s(k-1)\converge$; 
      
      \item \label{item:atomless:def_certified:ready:node_empty}
      $k^{\rho}_s(\beta)\diverge$; 
      and
      
      \item \label{item:atomless:def_certified:ready:2k_balls_exist} 
      Letting $v = u^{\alpha,\rho}_s(k-1)$ if $k>1$, $v=0$ otherwise, 
      $\SSS{\beta,\rho}{s}$ contains at least $2k$ many balls greater than~$v$. 

      % at least $2k$ many balls, each of which is greater than 
       % $\ge f^{\alpha,\rho}_s(k-1)$. 
    \end{orderedlist}

    \item \label{item:atomless:def_certified:ready:certified}
    We say that~$\beta$ is \emph{certified} at stage~$s$, if for all $\rho\in \{0,1\}^e$:
     \begin{orderedlist}[resume]
        \item $k^\rho_s(\beta)\converge$; and 
        \item $\domin_s(k)> f^{\alpha,\rho}_s(k)$, where $k=k^\rho_s(\beta)$. 
        %\item for the least such~$k$, .\footnote{Note that there can be more than one interval $I^{\alpha,\rho}_s(k)$ waiting at~$\beta$ at stage~$s$. This may happen when some interval is released by~$\beta$, and then existing intervals (with larger index~$k$) are pulled to~$\beta$ from nodes that lie to the right of~$\beta$. This can happen without any $A$-changes, in which case we are not allowed to dissolve existing intervals; but we are allowed to pull them to~$\beta$.} 
    \end{orderedlist}
  \end{sublemma}
  \end{definition}

  % Finally, we need to ensure that straggling balls do not get stuck too high.

  % \begin{definition} \label{def:atomless:decision_pulling}
  %   Let $\gamma^+$ be the unique child of a child $e$-splitting node~$\gamma$ (in other words, $\gamma^+$ is an $(e+1)$-decision node). Let $\rho\in \{0,1\}^{e}$; let~$s$ be a stage. We say that a ball $x$ is \emph{pullable} by $(\gamma^+,\rho\conc 0)$ at stage~$s$ if:
  %   \begin{orderedlist}
  %       \item $x\in \SSS{\gamma,\rho}{s}$;  and
  %       % \item $x>|\gamma|$; and
  %       % \item $x< \max \Y{\gamma,\rho}{s}$. 
  %       \item there is some~$k$ such that $f^{\alpha,\rho}_s(k)\converge$, $\gamma = \beta^{\alpha,\rho}_s(k)$, $I^{\alpha,\rho}_s(k)$ is already released from $\gamma$ at stage~$s$, and $x\in I^{\alpha,\rho}_s(k)$; here $\alpha = \gamma^-$.
  %   \end{orderedlist}
  % \end{definition}

\begin{lemma} \label{lem:atomless:there_is_a_strongest_pulling_node}
  Suppose that $x$ is pullable by some node at stage~$s$. Then:
  \begin{sublemma}
    \item For each node~$\beta$, there is at most one label~$\rho$ such that $x$ is pullable by $(\beta,\rho)$. 
    \item There is a $\le$-least node~$\beta$ such that $x$ is pullable by $(\beta,\rho)$ for some~$\rho$.
  \end{sublemma}
\end{lemma}

\begin{proof}
  For~(a), suppose that~$x$ is pullable by $(\beta,\rho)$. Then either $\beta$ is a child of an $e$-decision node or a child $e$-splitting node, $|\rho|=e$, and $x\in \Y{\beta^-,\rho}$. So~$\rho$ is the initial segment of length~$e$ of $x$'s label at stage~$s$. 

  (b) is as in \cref{lem:maximal:leftmost_pullable_node}. 
\end{proof}

\begin{lemma} \label{lem:atomless:uniqueness_of_ready_for_definition}
  Let $\alpha$ be a parent $e$-splitting node, and let $\rho\in \{0,1\}^e$. Let~$s$ be a stage. Suppose that no balls are pullable by any node at stage~$s$. Then there is at most one child~$\beta$ of $\alpha$ and one~$k$ such that $(\beta,\rho,k)$ is ready for definition at stage~$s$. 
\end{lemma}

\begin{proof}
 The number~$k$ is unique by \ref{item:atomless:def_certified:ready:k-1} of \cref{def:atomless:certified}. Say $\beta<\beta'$ are children of~$\alpha$. If $k^{\rho}_s(\beta)\diverge$, then every $x\in \Y{\beta',\rho}{s}$ is pullable by $(\beta,\rho)$; so by assumption on~$s$, $\Y{\beta',\rho}{s} =\emptyset$, whence $(\beta',\rho,k)$ cannot be ready for definition at~$s$ by \ref{item:atomless:def_certified:ready:2k_balls_exist}. 
\end{proof}

\medskip

\subsection{Construction} % (fold)

We start with $H_0 =\emptyset$ and no ball on the machine. 

\smallskip

At stage~$s$ we operate according to the first case that applies. 

\begin{sublemma}
	\item \label{item:atomless:construction:grabbing}
	There is a ball on the machine which is pullable by some $(\beta,\rho)$. For each such ball~$x$, let~$\beta$ be the $\le$-least such. We move $x$ to reside at~$\beta$ at stage $s+1$, and we set $x$'s label at stage $s+1$ to be~$\rho$. 
  %\footnote{It is possible that $x$'s label at stage~$s$ was an extension of~$\rho$, if~$x$ resided to the right of~$\beta$; or that $x$'s label was $\rho^-$, when~$\beta$ is the child of a child splitting node, as in \cref{def:atomless:decision_pulling}.} 

  % Also, if $\beta$ is the child of a parent $e$-splitting node~$\alpha$, no $(\alpha,\rho)$-interval is waiting at~$\beta$ at stage~$s$, $k\ge 1$, $I^{\alpha,\rho}_s(k)\converge$ and $\beta^{\alpha,\rho}_s(k)$ lies to the right of~$\beta$, then we redefine $\beta^{\alpha,\rho}_{s+1}(k)=\beta$, and declare that $I^{\alpha,\rho}_s(k)$ is waiting at~$\beta$ at stage $s+1$.\footnote{We do this even if no balls are actually pulled by~$\beta$ at this stage. Note that every element of $I^{\alpha,\rho}_s(k)\cap \Y{\beta^{\alpha,\rho}_s(k),\rho}{s}$ will be pulled by~$\beta$, but it is possible that this set is empty.} 

	\item \label{item:atomless:construction:eliminating}
	There is a ball~$x$ which is eliminable by some pair $(\beta,\rho)$. We enumerate all such~$x$ into~$H_{s+1}$ and remove them from the machine. 

  \item \label{item:atomless:construction:releasing}
  There is some child $e$-splitting node that is certified at stage~$s$. In this case, for each such~$\beta$ for which there is no $\Left{\gamma}{\beta}$ that is also certified at~$s$, for each $\rho\in \{0,1\}^e$, letting $k = k^\rho_s(\beta)$, we:
  % for the~$k$ such that $I^{\alpha,\rho}_s(k)\converge$ is waiting at~$\beta$ at stage~$s$ (where $\alpha = \beta^-$), we:
    \begin{itemize}
      \item Move all $x\in\SSS{\beta,\rho}{s}$
      %$x\in\Y{\beta,\rho}{s}\setminus \Y{\le\!\beta^+}{s}$ 
      to the unique child~$\beta^+$ of~$\beta$;
      \item For all children $\gamma\ge \beta$ of~$\alpha$, we set $k^{\rho}_{s+1}(\gamma)\diverge$; 
      \item For the $k$-least many balls $x\in\SSS{\beta,\rho}{s}$, we change their label to $\rho\conc 0$; for all other balls just moved to~$\beta^+$, we change the label to $\rho\conc 1$. 
    \end{itemize}

	\item \label{item:atomless:construction:creating_X}
	There is some child $e$-splitting node~$\beta$, some $\rho\in \{0,1\}^e$, and some $k\ge 1$, such that $(\beta,\rho,k)$ is {ready for definition} at stage~$s$. 

  For each such pair $(\beta,\rho)$, for the unique such~$k$, letting $\alpha = \beta^-$, we define 
   \[
   f^{\alpha,\rho}_{s+1}(k) = \domin_s(k)+1. 
   \]
  Let $v = u^{\alpha,\rho}_s(k-1)$ if $k>1$, $v=0$ otherwise. Let $x_1,x_2,\dots,$ enumerate (in order) the elements of $\SSS{\beta,\rho}{s}$ that are greater than~$v$. We let the $A_{s}$-use $u^{\alpha,\rho}_{s+1}(k)$ of the new computation be $x_{2k}+1$. 

  %  we let the use $u$ be the least such that $u\ge u^{\alpha,\rho}_s(k-1)$ (when $k>1$), and 
  % \[
  %   |\SSS{\beta,\rho}{s}\cap u| \ge 2k.\footnote{That is, $u= \max \{ u^{\alpha,\rho}_s(k-1),x+1\}$, where $x$ is the $2k\tth$ ball in $\SSS{\beta,\rho}{s}$, enumerated in order.}
  % \]
  We set $k^{\rho}_{s+1}(\beta) = k$.

	\item If no case above applies, then $s$ is declared to be a new balls stage. 
	\begin{orderedlist}
		\item We let $H_{s+1} = H_s\cup A_{s+1}$, and remove any $x\in A_{s+1}$ from the machine. 
		\item We place any $x\in Q_{s+1}\setminus H_{s+1}$ not already on the machine at the root of the machine, and give it the empty string as a label. 
    \item For any $e$-splitting child node~$\beta$, a child of some~$\alpha$, if $k^\rho_s(\beta)\converge = k$ but $f^{\alpha,\rho}_{s+1}(k)\diverge$, then $k^\rho_{s+1}(\beta)\diverge$. 
	\end{orderedlist}
\end{sublemma}

\medskip

\subsection{Verification} 

We start with some observations on ball movements and labels, that follow from the instructions. 

\begin{lemma} \label{lem:atomless:basic_label_facts}
  Let $\alpha$ be a node, $s$ be a stage, and suppose that $x\in \SSS{\alpha}{s}$. 
  \begin{sublemma}
    \item If $t>s$ and $x\in \SSS{\beta}{t}$, then $\beta\le \alpha$ (in the Kleene-Brouwer ordering).
    \item If $t>s$ and $x\in \SSS{\alpha}{t}$ (that is, $x$ has not moved between stages~$s$ and~$t$), then $x$'s label at stage~$s$ is the same as $x$'s label at stage~$t$. 
    \item If $\alpha$ is an $e$-decision or $e$-splitting node (parent or child), then $x$'s label at stage~$s$ has length~$e$. 
    \item If $t>s$ and $x\in \Y{\alpha}{t}$ then $x$'s label at stage~$t$ extends $x$'s label at stage~$s$. 
  \end{sublemma}
\end{lemma}

The proof of \cref{lem:maximal:Y_subset_of_Q} gives:

\begin{lemma} 
	For all~$s$, $\Y{\emptystring}{s} =  Q_s\setminus H_s$. 
\end{lemma}

\begin{lemma} \label{lem:bound_on_ball_size}
  If a ball~$x$ resides at a node~$\alpha$ at some stage, then $|\alpha| \le x+2$. 
\end{lemma}

\begin{proof}
  If $\alpha$ is the root then this is immediate. Otherwise, let $\eta\preceq \alpha$ be the longest node which is the child of a decision node. Each node only receives balls that have already passed through their parent. Hence, there is a stage at which~$x$ was pulled by~$\eta$. By \cref{def:atomless:decision_pullable}\ref{item:atomless:pull:size}, $x>|\eta|$; and $|\alpha|\le |\eta|+2$. 
\end{proof}

By \cref{lem:atomless:basic_label_facts}(a) and \cref{lem:bound_on_ball_size}, as in the proof of  \cref{lem:maximal:ball_movement_finite_part_1}, we get:

\begin{lemma}
	If $x\in \bigcup_s \Y{\emptystring}{s}$ and $x\notin H$ then~$x$ is a permanent resident of some node.
\end{lemma}

\begin{lemma} \label{lem:atomless:infinitely_many_new_balls}
  There are infinitely many new ball stages. 
\end{lemma}

\begin{proof}
  The point is that if not, then case~(d) of the construction (defining new blocks) can happen at most finitely often. In more detail: suppose, for a contradiction, that $s_0$ is the last new balls stage. There is some stage $s_1>s_0$ after which no balls are moved on the machine, or eliminated from the machine. If a new computation $f^{\alpha,\rho}_{s+1}(k)$ is defined at some stage $s>s_1$, then this computation is never made undefined. For every child~$\beta$ of~$\alpha$, for all $t>s$, if $\SSS{\beta,\rho}{t}\ne \emptyset$ then $k^\rho_t(\beta)\converge$ (at stage~$s$, the fact that a new computation is defined, implies that there are no children of~$\alpha$ which have balls, but insufficiently many to form a block: such balls would be pulled away earlier than stage~$s$). Hence, no further values of $f^{\alpha,\rho}$ will be defined after stage~$s$. 
\end{proof}

Since every element of $\outside{A}$ is added to the root of the tree at some stage, it follows that every $x\in \outside{H}$ is a permanent resident of some node.

\subsubsection{Defining the true path}

We define the true path inductively. The root $\emptystring$ is declared to lie on the true path. Let $\alpha$ be a node on the true path. 
\begin{itemize}
  \item If $\alpha$ is a decision node, then the leftmost child $\beta$ of~$\alpha$ with $\seqell{\beta}$ unbounded is on the true path.
  \item If $\alpha$ is a parent splitting node, then the leftmost child~$\beta$ of~$\alpha$ that releases balls infinitiely often, lies on the true path. 
  \item If $\alpha$ is a child splitting node, then its unique child is also on the true path. 
\end{itemize}
As above, if $\alpha$ is a decision node, then as one of the statements $\psi(\alpha,D)$ is true, one of $\alpha$'s children will lie on the true path. On the other hand, we will need to work to show that if $\alpha$ is a parent splitting node on the true path, then one of its children lies on the true path. For now, we do not know that the true path is infinite. 

In fact, to show that a parent splitting node~$\alpha$ which lies on the true path has a child on the true path, we will need to already know that the analogues of \cref{lem:maximal:left_of_true_path,lem:maximal:main_true_path_lemma} hold for~$\alpha$. Thus, we will need to prove both by simultaneous induction.

Recall (\cref{def:Q_and_C}) that we let
  \[
    C_s = \set{x\in Q_s}{A_s\rest{x+1} = A\rest{x+1}},
  \]
  the collection of balls $x\in Q_s$ which are $A$-true at stage~$s$. 

\begin{proposition} \label{prop:atomless:main_true_path_lemma}
  Suppose that a node~$\alpha$ lies on the true path. 
  \begin{sublemma}
  \item  There is a stage after which for every node $\gamma$ that lies to the left of~$\alpha$, $\gamma$ does not eliminate any balls from $\outside{A}$, nor is any ball from $\outside{A}$ moved to~$\gamma$.\footnote{Balls can move to~$\gamma$ without being pulled by~$\gamma$, if $\gamma$ is the child of a child splitting node, and balls are released by $\gamma^-$.}
  \item  Suppose that~$\alpha$ is an $e$-splitting or decision node. Then for all $\rho\in \{0,1\}^e$, for every $k$, there is some stage~$s$ such that $|\Y{\alpha,\rho}_s\cap C_s|\ge k$. 
  \end{sublemma}
\end{proposition}

The proof of \cref{prop:atomless:main_true_path_lemma} will take some work. As discussed, we prove it by induction on the length of~$\alpha$. For now, we note:

\begin{lemma} \label{lem:atomless:finitely_many_permanent_balls_to_the_left}
  Let~$\alpha$ be a node; 
  suppose that there is a stage after which no ball is moved to a node $\Left{\gamma}{\alpha}$. Then $\Y{<_{L}\!\alpha}$ is finite. 
\end{lemma}

\begin{proof}
  Every ball in $\Y{<_{L}\!\alpha}$ is in $\outside{A}$, and is at some stage moved to  a node~$\gamma$ that lies to the left of~$\alpha$. By assumption, there are only finitely many such stages, and at each stage, only finitely many balls are on the machine. 
\end{proof}

\subsubsection{Case I: the root}

We start the inductive verification of \cref{prop:atomless:main_true_path_lemma} by verifying that it holds for the root. (a) is vacuous in this case;  (b) is as in \cref{lem:maximal:main_true_path_lemma}. We note that this is the only part of the proof in which we use the fact that we are using the true-stage enumeration of~$A$.

\subsubsection{Case II: decision nodes}

Suppose that $\alpha$ is an $e$-decision node that lies on the true path, and suppose that \cref{prop:atomless:main_true_path_lemma} holds for~$\alpha$. Let~$\beta$ be $\alpha$'s child that lies on the true path; we show that \cref{prop:atomless:main_true_path_lemma} holds for~$\beta$ as well. Let $D\subseteq \{0,1\}^e$ such that $\beta = \alpha\conc D_n$ for some~$n$. That is, $\rho\in D$ if and only if for all~$k$ there is some~$s$ such that $|\Y{\alpha,\rho}{s}\cap W_{e,s}\cap C_s|\ge k$.

\medskip

To show that~(a) holds for~$\beta$, we only need a minor modification of the proof \cref{lem:maximal:left_of_true_path}, reflecting the new second part of \cref{def:atomless:decision_pullable}\ref{item:atomless:pull:new}.

Let $l$ be the maximum of $\lim_s \ell(\gamma)_s$, where~$\gamma$ is a child of~$\alpha$ that lies to the left of~$\beta$. Again, for any such~$\gamma$ and any $\epsilon\succeq \gamma$ that is a child of some decision node, $\lim_s \ell(\epsilon)_s \le l$. By \cref{def:atomless:decision_pullable}(vii), no such~$\epsilon$ ever eliminates any $x\ge l$. 

Let $\gamma$ be a child of~$\alpha$ that lies to the left of~$\beta$, and suppose (by induction) that every child $\delta$ of~$\alpha$ that lies to the left of~$\gamma$ eventually stops pulling any $x\in \outside{A}$ (and so, from some stage, no ball from $\outside{A}$ is moved to any node that lies to the left of~$\gamma$). 

Let $\rho\in \{0,1\}^e$. For a contradiction, suppose that $(\gamma,\rho)$ pulls infinitely many balls from $\outside{A}$. By our assumption, almost all such balls remain in $\Y{\gamma}$. By \cref{lem:atomless:basic_label_facts}(d), the balls pulled by $(\gamma,\rho)$ that remain in $\Y{\gamma}$ also remain in $\Y{\gamma,\rho}$. Let $s$ be a late stage at which $\Y{\gamma,\rho}{s}$ contains at least $l$ many elements of $\outside{A}$; these balls will not be eliminated, so for all $t\ge s$, $|\Y{\gamma,\rho}{t}|\ge l$. Let $a = \max \Y{\gamma,\rho}{s}$. Then by induction on $t\ge s$ we show that $a = \max \Y{\gamma,\rho}{t}$: if this holds for~$t$, then by \ref{item:atomless:pull:new} of \cref{def:atomless:decision_pullable}, the only numbers pulled by $(\gamma,\rho)$ are $<a$. 

\medskip

The proof that~(b) holds for~$\beta$ is similar to that of \cref{lem:maximal:main_true_path_lemma}, but here we use the new \ref{item:atomless:pull:new} of \cref{def:atomless:decision_pullable}. Fix $\rho\in \{0,1\}^e$. By~(a), let~$s_0$ be a stage after which no ball from $\outside{A}$ is moved to, or eliminated by, any node $\Left{\gamma}{\beta}$.

By \cref{lem:atomless:finitely_many_permanent_balls_to_the_left}, $\Y{<_{L}\!\beta,\rho}$ is finite; let $m = |\Y{<_{L}\!\beta,\rho}|$. If $s>s_0$ and $x\in \Y{<_{L}\!\beta}{s}\cap C_s$, then as~$x$ will not be later eliminated (and $x\notin A$), $x\in \Y{<_{L}\!\beta}$. By \cref{lem:atomless:basic_label_facts}(b), for all $s>s_0$, $\Y{<_{L}\!\beta,\rho}{s}\cap C_s \subseteq \Y{<_{L}\!\beta,\rho}$. Hence, for all $s>s_0$, $|\Y{<_{L}\!\beta,\rho}{s}\cap C_s| \le m$. 

Let $k\in \Nat$. Suppose that $\rho\notin D$. By induction, there is some $s>s_0$ such that $|\Y{\alpha,\rho}{s}\cap C_s| \ge k+m+3$ and such that $\oneell{\beta}{s}>k$. Let
\[
  R = \left\{ x\in C_s\cap \Y{\alpha,\rho}{s} \,:\,  x\notin \Y{\le\!\beta,\rho}{s}\andd  x>|\alpha|\right\}. 
\]
Since $x\ge |\alpha|-2$ for all $x\in \Y{\alpha}{s}$ (\cref{lem:bound_on_ball_size}), $|R|\ge k - |\Y{\beta,\rho}{s}\cap C_s|$. So as in the maximal set construction, we will be done once we show that either $|\Y{\beta,\rho}{s}\cap C_s|\ge k$, or that every $x\in R$ is pullable by $(\beta,\rho)$ at stage~$s$ (note that in the latter case, $s$ will not be a new balls stage, so $C_s = C_{s+1}$). Suppose that $|\Y{\beta,\rho}{s}\cap C_s|< k$. The new part is that it is possible, in this case, that  $|\Y{\beta,\rho}{s}|\ge \oneell{\beta}{s}$. If $|\Y{\beta,\rho}{s}|< \oneell{\beta}{s}$ then certainly all balls in~$R$ are pullable. If not, though, since $\oneell{\beta}{s}>k$, there must be some $z\in \Y{\beta,\rho}{s}$ which is not in~$C_s$. However, $C_s\setminus H_s$ is an initial segment of the balls on the machine at stage~$s$. So all balls in~$R$ are smaller than~$z$. By the new part of \cref{def:atomless:decision_pullable}\ref{item:atomless:pull:new}, all balls in~$R$ are pullable by $(\beta,\rho)$ at stage~$s$. 

The case $\rho\in D$ is the same, using the fact that as $\psi(\alpha,D)$ holds; there is some $s>s_0$ such that $\oneell{\beta}{s}>k$, and $|\Y{\alpha,\rho}{s}\cap C_s\cap W_{e,s}| \ge k+m+3$. This completes the proof that \cref{prop:atomless:main_true_path_lemma} holds for~$\beta$.

\subsubsection{Case III: splitting nodes}

For the rest of the proof of \cref{prop:atomless:main_true_path_lemma}, let~$\alpha$ be a parent $e$-splitting node that lies on the true path, and suppose that the proposition holds for~$\alpha$. We show that~$\alpha$ has a child~$\beta$ on the true path, and that the proposition holds for the unique child~$\beta^+$ of~$\beta$ (and hence also for~$\beta$). 

Until the end of the proof of \cref{prop:atomless:main_true_path_lemma}, we fix a stage~$s_0$ witnessing that \cref{prop:atomless:main_true_path_lemma}(a) holds for~$\alpha$.

\begin{lemma} \label{lem:atomless:block_location}
  Let $\rho \in \{0,1\}^e$; let~$s$ be a stage. 
  \begin{sublemma}
  \item \label{item:lem:atomless:block_location:k-1}
  Let $k> 1$.  If $f^{\alpha,\rho}_s(k)\converge$ then  $f^{\alpha,\rho}_{s}(k-1)\converge$. 

  \item \label{item:lem:atomless:block_location:k_then_f}
  If $\beta$ is a child of~$\alpha$ and $k^\rho_s(\beta)\converge$, then $f^{\alpha,\rho}_s(k^\rho_s(\beta))\converge$. 

  \item \label{item:lem:atomless:block_location:beta_left}
  If $\beta$ is a child of~$\alpha$ and $k^\rho_s(\beta)\converge$, then for every child $\gamma<\beta$ of~$\alpha$, $k^\rho_s(\gamma)\converge$ and $k^\rho_s(\gamma) < k^\rho_s(\beta)$. 
  \end{sublemma}
\end{lemma}

\begin{proof}
  \ref{item:lem:atomless:block_location:k-1} follows from the fact that when we define a new computation $f^{\alpha,\rho}_{s+1}(k)$, we have $f^{\alpha,\rho}_s(k-1)\converge$, and we set $u^{\alpha,\rho}_{s+1}(k)>u^{\alpha,\rho}_{s}(k-1)$. Also, we note that such a stage~$s$ is not a new balls stage, so $A_s=A_{s+1}$, so $u^{\alpha,\rho}_{s}(k-1) = u^{\alpha,\rho}_{s+1}(k-1)$. 

  \ref{item:lem:atomless:block_location:k_then_f} is by our stipulation that when an $A$-change causes a computation $f^{\alpha,\rho}_s(k)$ to become undefined at stage $s+1$, and $k = k^{\rho}_s(\beta)$, then we set $k^\rho_{s+1}(\beta)\diverge$. 

  \ref{item:lem:atomless:block_location:beta_left} follows from the previous two parts, and the fact that a new computation is always set up at the leftmost ``free'' child (a child with $k^\rho_s(\beta)\diverge$); see the proof of \cref{lem:atomless:uniqueness_of_ready_for_definition}.
\end{proof}

Recall that we say that a computation $f^{\alpha,\rho}_s(k)\converge$ is \emph{$A$-correct} if $A_s\rest{u} = A\rest{u}$, where $u=u^{\alpha,\rho}_s(k)$ is the use of the computation $f^{\alpha,\rho}_s(k)$. Also recall that $f^{\alpha,\rho}_s(k)$ is $A$-correct if and only if for all $t\ge s$, $f^{\alpha,\rho}_t(k)\converge$. 

\begin{lemma} \label{lem:atomless:block_location:block_size}
  Let~$\beta$ be a child of~$\alpha$; let $\rho\in \{0,1\}^e$; let~$s$ be a stage. Suppose that $k=k^\rho_s(\beta)\converge$, and suppose that the computation $f^{\alpha,\rho}_s(k)$ was defined after stage~$s_0$, and is $A$-correct. Then $\SSS{\beta,\rho}{s}$ contains $2k$ many balls $x< u^{\alpha,\rho}_s(k)$. 
\end{lemma}

\begin{proof}
  Let~$r$ be the stage at which the computation was defined. Let $X$ be the set of $x\in \SSS{\beta,\rho}{r}$ such that $x<u^{\alpha,\rho}_{r+1}(k)$. Then by the construction, $|X|\ge 2k$, and by assumption that the computation is $A$-correct, $X\subseteq \outside{A}$.  By induction on $t\in [r,s]$ we see that $X\subseteq \SSS{\beta,\rho}{t}$. Suppose that this holds for~$t$, and $t<s$. Since $t>s_0$, no ball from~$X$ is pulled by a node to the left of~$\alpha$. If $\gamma<\beta$ is a child of~$\alpha$, and any ball from~$X$ is moved to $\Y{\gamma}{t+1}$, then $k^\rho_{t+1}(\gamma)\diverge$; by \cref{lem:atomless:block_location}, $k^\rho_{t+1}(\beta)\diverge$. But then, since $f^{\alpha,\rho}_t(k)$ is $A$-correct, there is no stage $w>t$ at which $k^\rho_w(\beta)=k$, contradicting the hypothesis of this lemma. Hence, $X\subseteq \SSS{\beta}{t+1}$. 
\end{proof}

The following follows from the construction.

\begin{lemma} \label{lem:equivalences_of_stability}
  Let~$\beta$ be a child of~$\alpha$, and let $\rho\in \{0,1\}^e$. Let $s\ge s_0$ be a stage. 
  \begin{sublemma}
    \item Suppose that for all $t\ge s$, $k^\rho_t(\beta)\converge$. Then $f^{\alpha,\rho}_s(k^\rho_s(\beta))$ is $A$-correct, and for all $t\ge s$, $k^\rho_t(\beta)=k^\rho_s(\beta)$, and~$\beta$ does not release balls at stage~$t$. 

    \item Suppose that $s\ge s_0$, $k^\rho_s(\beta)\converge$, and $f^{\alpha,\rho}_s(k^\rho_s(\beta))$ is $A$-correct. Suppose, further, that for all children $\gamma\le \beta$ of~$\alpha$, for all $t\ge s$, $\gamma$ does not release balls at stage~$t$. Then for all $t\ge s$, $k^\rho_t(\beta)\converge$. 
  \end{sublemma}
\end{lemma}

\begin{definition} \label{def:k_beta_is_stable}
  Let~$\beta$ be a child of~$\alpha$, and let $\rho\in \{0,1\}^e$. We write $k^\rho(\beta)\converge$ if $k^\rho_s(\beta)\converge$ for all but finitely many~$s$. 
\end{definition}

The following is the main lemma.

\begin{lemma} \label{lem:atomless:splitting:main}
	Let~$\beta$ be a child of~$\alpha$, and suppose that for every every child $\gamma<\beta$ of~$\alpha$, for every $\rho\in \{0,1\}^e$,  $k^\rho(\gamma)\converge$. Then for all $\rho\in \{0,1\}^{e}$, there are infinitely many stages~$t$ such that $k^\rho_t(\beta)\converge$ and $f^{\alpha,\rho}_t(k^\rho_t(\beta))$ is $A$-correct. 
\end{lemma}

An immediate consequence is the following:

\begin{lemma} \label{lem:atomless:splitting:main:consequence}
  Let~$\beta$ be a child of~$\alpha$, and suppose that for every every child $\gamma<\beta$ of~$\alpha$, for every $\rho\in \{0,1\}^e$,  $k^\rho(\gamma)\converge$. Then either:
    \begin{equivalent}
  \item For all $\rho\in \{0,1\}^{e}$, $k^\rho(\beta)\converge$; or

  \item $\beta$ releases balls infinitely often, in fact, for all $\rho\in \{0,1\}^{e}$ there are infinitely many~$t$ such that $\beta$ releases balls at stage~$t$, and $f^{\alpha,\rho}_t(k^\rho_t(\beta))$ is $A$-correct. 
  \end{equivalent}
\end{lemma}

\begin{proof}
  Let~$s_1\ge s_0$ be a stage witnessing the assumption on the children $\gamma<\beta$ of~$\alpha$: for all $s\ge s_1$, for all $\rho\in \{0,1\}^e$, $k^\rho_s(\gamma)\converge$. 

  Let $\rho\in \{0,1\}^e$. Let $s\ge s_1$ be a stage such that $k^\rho_s(\beta)\converge$, and $f^{\alpha,\rho}_s(k^\rho_s(\beta))$ is $A$-correct. If $\beta$ does not release any balls after stage~$s$, then by \cref{lem:equivalences_of_stability}(b), stage~$s$ shows that $k^\rho(\beta)\converge$. So if $\beta$ releases balls only finitely many times, (1) holds for~$\beta$. 

  Suppose that $\beta$ releases balls infinitely often; let $t$ be the least stage $\ge s$ at which~$\beta$ releases balls. The assumption that $f^{\alpha,\rho}_s(k^\rho_s(\beta))$ is $A$-correct, and $s\ge s_1$, imply that $k= k^\rho_t(\beta) = k^\rho_s(\beta)$ and that $f^{\alpha,\rho}_t(k)\converge$ is $A$-correct, so $t$ is one of the stages as required for~(2). 
\end{proof}

\begin{proof}[Proof of \cref{lem:atomless:splitting:main}]
 Let~$s_1\ge s_0$ be a stage witnessing the assumption on the children $\gamma<\beta$ of~$\alpha$.
Fix some $\rho\in \{0,1\}^{e}$. Let $s_2\ge s_1$ be any stage. We show that there is some $t\ge s_2$ such that $k^\rho_t(\beta)\converge$ and $f^{\alpha,\rho}_t(k^\rho_t(\beta))$ is $A$-correct. Suppose, for a contradiction, that this is not the case. In other words:

\begin{description}
  \item[$\otimes_1$] If $s\ge s_2$ and $k^\rho_s(\beta)\converge=k$, then there is some $t>s$ such that $f^{\alpha,\rho}_t(k)\diverge$. 
\end{description}

Let $m$ be the least such that there is some $s\ge s_2$ such that $f^{\alpha,\rho}_s(m)\diverge$. 
For all $k<m$, $f^{\alpha,\rho}_{s_2}(k)\converge$ is $A$-correct. Hence, our assumption for contradiction implies:

\begin{description}
  \item[$\otimes_2$] For all $s\ge s_2$, either $k^\rho_s(\beta)\diverge$, or $k^\rho_s(\beta)\ge m$. 
\end{description}

% By the choice of~$m$, there is some stage $s\ge s_2$ such that $f^{\alpha,\rho}_{s}(m)\diverge$. If $k^\rho_s(\beta)\converge$, then by $\otimes_2$, $k^\rho_s(\beta)\ge m$. But then by \cref{lem:atomless:block_location} we would have $f^{\alpha,\rho}_{s}(m)\converge$. Hence:

% \begin{description}
%   \item[$\otimes_3$] There is some $s\ge s_2$ such that $k^\rho_s(\beta)\diverge$. 
% \end{description}

% Also:

% \begin{description}
%   \item[$\otimes_4$] Suppose that $s\ge s_2$ and $k^\rho_s(\beta)\diverge$. Suppose that there is some $t>s$ such that $k^\rho_t(\beta)\converge$. Then for the least such~$t$, we have $k^\rho_t(\beta)=m$. [[[Not true]]]
% \end{description}

We claim:

\begin{description}
  \item[$\otimes_3$] There are infinitely many stage~$s$ at which $f^{\alpha,\rho}_s(m)\diverge$. 
\end{description}

Suppose, for a contradiction, that this is not the case. Thus, $f^{\alpha,\rho}(m)\converge$. Let~$r$ be least such that $f^{\alpha,\rho}_r(m)$ is $A$-correct, i.e., the correct computation $f^{\alpha,\rho}(m)$ is defined at stage $r-1$. By the choice of~$m$, $r>s_2$. The computation is established at some child~$\gamma$ of~$\alpha$: $k^\rho_r(\gamma)=m$. Since $r>s_1$, $\gamma \ge \beta$. By \cref{lem:atomless:block_location}, if $\gamma>\beta$ then $k^\rho_r(\beta)\converge<m$, contradicting $\otimes_2$. Hence, $\gamma=\beta$. But then, by $\otimes_1$, there is some $s>r$ such that $f^{\alpha,\rho}_s(m)\diverge$, giving the desired contradiction that establishes~$\otimes_3$. 

We also observe:

\begin{description}
  \item[$\otimes_4$] If $s\ge s_2$ and $f^{\alpha,\rho}_s(m)\diverge$ then $k^\rho_s(\beta)\diverge$. 
\end{description}

This follows from \cref{lem:atomless:block_location} and~$\otimes_2$: if $k=k^\rho_s(\beta)\converge$ then $f^{\alpha,\rho}_s(k)\converge$ and $k\ge m$, so $f^{\alpha,\rho}_s(m)\converge$. 

\smallskip

There are only finitely many stages~$s$ such that $k^\rho_s(\beta) = m$ and~$\beta$ releases balls at stage~$s$ (this is the main point of using the certification process). To see this, let $s$ be such that $k^\rho_s(\beta)=m$ and~$\beta$ releases balls at stage~$s$.  Let~$r$ be the stage at which the computation $f^{\alpha,\rho}_s(m)$ was defined. Then $\domin_r(m) < f^{\alpha,\rho}_s(m) < \domin_s(m)$, in particular, $\domin_r(m)\ne \domin_s(m)$. Thus, there is at most one such stage~$s$ after the last stage at which $\domin_t(m)$ changes. 

Let $s_3\ge s_2$ be a stage after which there are no such stages~$s$; by $\otimes_3$, assume that $f^{\alpha,\rho}_{s_3}(m)\diverge$.

\begin{description}
  \item[$\otimes_5$] For all $s\ge s_3$, either $f^{\alpha,\rho}_s(m)\diverge$ or $k^\rho_s(\beta)=m$. 
\end{description}

We prove~$\otimes_5$ by induction on $s\ge s_3$. By assumption, this holds for $s=s_3$. Let $s\ge s_3$. First, suppose that $f^{\alpha,\rho}_s(m)\diverge$. By $\otimes_4$, $k^\rho_s(\beta)\diverge$. Hence, if a new computation $f^{\alpha,\rho}_{s+1}(m)$ is defined at stage~$s$, it is established at~$\beta$, setting $k^\rho_{s+1}(\beta)=m$. If not, then $f^{\alpha,\rho}_{s+1}(m)\diverge$. Next, suppose that $k^\rho_s(\beta)=m$. Since $s\ge s_3$, $\beta$ does not release any balls at stage~$s$. Hence, either $k^\rho_{s+1}(\beta)=m$, or $f^{\alpha,\rho}_{s+1}(m)\diverge$. 

Putting $\otimes_5$ and the choice of~$s_3$ together, we obtain:

\begin{description}
  \item[$\otimes_6$] $\beta$ does not release balls at any stage $s\ge s_3$. 
\end{description}

By choice of $s_0$ and~$s_1$, after stage~$s_3$, no ball from $\outside{A}$  is moved to~$\beta^+$, or from $\beta^+$ to any node to tis left. So for all $s\ge s_3$, $\Y{\beta^+,\rho}{s} = \Y{\beta^+,\rho}$, so $\Y{\beta^+,\rho}$ is finite. Let $N = |\Y{\beta^+,\rho}|$. Let $v = u^{\alpha,\rho}(m-1)$ if $m>1$, $v=0$ otherwise. 

By the assumption that \cref{prop:atomless:main_true_path_lemma}(b) holds for~$\alpha$, let $s\ge s_3$ be a stage at which $|\Y{\alpha,\rho}_s\cap C_s|\ge N+ 2m+v$. Let $X$ be the set of $x\in \Y{\alpha,\rho}{s}\cap C_s$ such that $x\notin \Y{\beta^+}{s}$ and $x> v$. 

Then $|X|\ge 2m$, and for all $t\ge s$, $X\subseteq \Y{\alpha,\rho}{t}\setminus \Y{\beta^+}{t}$. By $\otimes_3$ and $\otimes_4$, let $t>s$ be a stage such that $k^\rho_t(\beta)\diverge$. Then by \cref{def:atomless:certified}, either $X\subseteq \SSS{\beta,\rho}_{t}$, or balls are pulled by~$\beta$ at stage~$t$, resulting in $X\subseteq \SSS{\beta,\rho}{t+1}$. In either case, we obtain a stage $r\ge s_3$ at which $f^{\alpha,\rho}_r(m)\diverge$, $k^\rho_r(\beta)\diverge$, and $X\subseteq \SSS{\beta,\rho}{r}$. So at stage~$r$ we define a new computation $f^{\alpha,\rho}_{r+1}(m)$ and set $k^\rho_{r+1}(\beta)=m$. The use $u^{\alpha,\rho}_{r+1}(m)$ is defined to be $x+1$ for some $x\in X$. Since $X\subseteq C_r$, this shows that $f^{\alpha,\rho}_{r+1}(m)$ is $A$-correct. This contradicts $\otimes_3$, which finishes the proof of \cref{lem:atomless:splitting:main}.
\end{proof}

\begin{lemma} \label{lem:atomless:splitting_child_on_true_path}
  Some child of~$\alpha$ lies on the true path. 
\end{lemma}

Recall that this means that some child of~$\alpha$ releases balls infinitely often. 

\begin{proof}
  Suppose, for a contradiction, that no child of~$\alpha$ releases balls infinitely often. By induction on the children of~$\alpha$, from left to right, we see, using \cref{lem:atomless:splitting:main:consequence}, that for every child~$\beta$ of~$\alpha$, for every $\rho\in \{0,1\}^e$, $k^\rho(\beta)\converge$. Fixing~$\rho$, since the map $\beta\mapsto k^\rho(\beta)$ is injective (\cref{lem:atomless:block_location}), we see that $f^{\alpha,\rho}$ is total. 

  Since each $f^{\alpha,\rho}$ is $A$-computable, $\domin$ dominates $\max_{\rho} f^{\alpha,\rho}$. Say this domination starts at $k^*$. Let~$\beta$ be a child of~$\alpha$ that lies sufficiently to the right, so that for all $\rho\in \{0,1\}^e$, $k^\rho(\beta)\ge k^*$. Let~$s$ be sufficiently late so that for each~$\rho$, for all~$t\ge s$, $k^\rho_t(\beta)\converge = k^\rho(\beta)$, and $\domin_s(k^\rho_t) > f^{\alpha,\rho}(k^\rho(\beta))$. But then, $\beta$ is certified at each stage $t\ge s$, and so, will release balls -- contradicting \cref{lem:equivalences_of_stability}.
\end{proof}

Let~$\beta$ be the child of~$\alpha$ that lies on the true path.

\begin{lemma} \label{lem:atomless:true_path_child_is_well_behaved}
  \Cref{prop:atomless:main_true_path_lemma} holds for both~$\beta$ and~$\beta^+$. 
\end{lemma}

\begin{proof}
  By \cref{lem:atomless:splitting:main:consequence}, for every child $\gamma<\beta$ of~$\alpha$, for all~$\rho$, $k^\rho(\gamma)\converge$. By \cref{lem:equivalences_of_stability}, eventually, each such child stops pulling any balls (whether from~$\outside{A}$ or not), and does not release any balls. Thus, for each such~$\gamma$, only finitely many balls ever enter $\Y{\gamma}$, and each such ball is eventually removed from the machine, or stops moving. This gives part~(a) of \cref{prop:atomless:main_true_path_lemma} for~$\beta^+$ (and hence for~$\beta$ as well). 

  Part~(b) follows from \cref{lem:atomless:splitting:main:consequence} as well. Fix $\rho\in \{0,1\}^e$. Let~$t$ be a stage. If $k^\rho_t(\beta)\converge=k$, $f^{\alpha,\rho}_t(k)$ is $A$-correct, and $\beta$ releases balls at stage~$t$, then by \cref{lem:atomless:block_location:block_size}, $|\SSS{\beta^+,\rho}{t+1}\cap C_{t+1}|\ge 2k$. By construction, in fact, for each $i=0,1$, $|\SSS{\beta^+,\rho\conc i}{t+1}\cap C_{t+1}|\ge k$. 

  If $t'>t$, $k^\rho_{t'}(\beta)\converge = k'$, and $f^{\alpha,\rho}_{t'}(k')$ is $A$-correct, then $k'>k$, showing that these numbers~$k$ go to~$\infty$.
\end{proof}

% \begin{lemma} \label{lem:atomless:true_path_child_is_well_behaved}
%   \Cref{prop:atomless:main_true_path_lemma} holds for the unique child $\beta^+$ of~$\beta$. 
% \end{lemma}

% \begin{proof}
%   Part~(a) of follows from the fact that it holds for~$\beta$, since $\beta^+$ is the only child of~$\beta$. Part~(b) also follows from the argument above. Since~$\beta^+$ is an $(e+1)$-decision node, we need to show that for each $\rho\in \{0,1\}^e$  and each $i\in \{0,1\}$, for all~$k$ there is some~$s$ such that $|\Y{\beta^+,\rho\conc i}{s}\cap C_s|\ge k$. By \cref{lem:atomless:splitting:main}, there are infinitely many~$k$ for which at some stage~$s$, $\beta$ releases $I^{\alpha,\rho}_s(k)$, and that interval is $A$-correct. As mentioned, for almost all such~$k$, the released block contains at least $2k$ many balls, all of which are in $C_s = C_{s+1}$. For each $i$, at least $k$ many balls receive the label $\rho\conc i$, and so $|\SSS{\beta^+,\rho\conc i}{s+1}\cap C_{s+1}|\ge k$. 
% \end{proof}

\begin{corollary} \label{cor:atomless:the_true_path_is_infinite}
  The true path is infinite, and \cref{prop:atomless:main_true_path_lemma} holds for every $\alpha$ on the true path. 
\end{corollary}

\bigskip

\subsubsection{The rest} % (fold) 

The rest is now straightforward. For each $\rho\in 2^{<\w}$, let 
\[
	Z(\rho)
\]
be the collection of balls $x\notin H$ whose permanent label extends~$\rho$. 

\begin{lemma} \label{lem:atomless:sets_nonempty}
	Let~$\alpha$ on the true path be an $e$-decision or splitting node. For every $\rho \in \{0,1\}^{e}$, $\Y{\alpha,\rho}\ne \emptyset$. 
\end{lemma}

\begin{proof}
	Very much like the proof of \cref{lem:maximal:permanent_residents_on_true_path}, except that replacing true stages by~$C_s$ actually makes things slightly simpler. We may assume that~$\alpha$ is an $e$-decision node. Let~$s_0$ be a stage after which nodes that lie to the left of~$\alpha$ do not eliminate any balls, or pull any balls in~$\outside{A}$.  Let
	\[
		x = \min \bigcup_{s>s_0} (C_s\cap \Y{\alpha,\rho}{s}).
	\]
	$x$ exists by \cref{prop:atomless:main_true_path_lemma}. Say $s_1>s_0$ and $x\in C_{s_1}\cap \Y{\alpha,\rho}{s_1}$. By induction on $t\ge s_1$ we see that $x = \min \Y{\alpha,\rho}{t}$. For $t = s_1$ this is because $C_{s_1}$ is an initial segment of~$Q_{s_1}$. Suppose that $x = \min \Y{\alpha,\rho}{t}$. Then~$x$ is not eliminable by any node $\gamma\succeq \alpha$, and so $x\in \Y{\alpha,\rho}{t+1}$. Let $y<x$. If $y\in \Y{\alpha,\rho}{t+1}$, then as $C_{s_1}\subseteq C_{t+1}$, we have $y\in C_{t+1}$, contradicting the minimality of~$x$. 
\end{proof}

\begin{corollary} \label{cor:atomless:sets_infinite}
	Let~$\alpha$ on the true path be an $e$-decision or splitting node. For every $\rho \in \{0,1\}^{e}$, $\Y{\alpha,\rho}$ is infinite. 
\end{corollary}

It follows that each $Z(\rho)$ is infinite.

\begin{lemma} \label{lem:atomless:all_balls_on_true_path}
	For every~$\alpha$ on the true path, $\Y{\alpha} =^* \outside{H}$. 
\end{lemma}

\begin{proof}
	Like the proof of \cref{lem:all_balls_on_true_path}. The new part is when~$\alpha$ is a child $e$-splitting node; we need to show that every ball that lies to the right of~$\alpha$ but below~$\alpha^-$ is eventually pullable by~$\alpha$. This follows from the fact that~$\alpha$ releases balls infinitely often; if $\alpha$ releases balls at stage~$s$, then $k^\rho_{s+1}(\alpha)\diverge$ for each~$\rho$, and then every such ball~$x$ is pullable by~$\alpha$ at stage $s+1$. 
\end{proof}

As a result, for all~$\rho$, if $\alpha$ is an $e$-decision or splitting node on the true path, then $Z(\rho) =^* \Y{\alpha,\rho}$. 

\begin{lemma} 
	For every~$\rho$, $H\cup Z(\rho)$ is c.e.
\end{lemma}

\begin{proof}
	Let~$\alpha$ be an $e$-decision node on the true path, where $|\rho| = e$. Then $\Y{\alpha,\rho}$ is c.e.: from some stage onwards, any ball outside~$H$ which enters~$\Y{\alpha}$ with label~$\rho$, will stay in $\Y{\alpha,\rho}$. 
\end{proof}

\begin{lemma} 
	For every~$\rho$, $Z(\rho) =^* Z(\rho\conc 0)\cup Z(\rho\conc 1)$, and $Z(\rho\conc 0)\cap Z(\rho\conc 1) = \emptyset$. 
\end{lemma}

\begin{proof}
	If $\alpha$ is an $e$-decision node on the true path with $e= |\rho|+1$, then $\Y{\alpha,\rho} = \Y{\alpha,\rho\conc 0} \cup \Y{\alpha,\rho\conc 1}$.
\end{proof}

\begin{lemma} 
	For every~$e$, $W_e\cap \outside{H}$ is the union of finitely many sets $Z(\rho)$, up to finite difference. 
\end{lemma}

\begin{proof}
	Like the proof of \cref{lem:maximality:maximality}. Let~$\alpha$ on the true path be a child of an $e$-decision node. Then for all $\rho\in \{0,1\}^{e}$, either $\Y{\alpha,\rho}\subseteq W_e$, or $\Y{\alpha,\rho}\cap W_e =^* \emptyset$. The result follows since  $Z(\rho)=^* Y(\alpha,\rho)$ for each such~$\rho$, and $Z(\rho)$ for $\rho\in \{0,1\}^e$ partition $\outside{H}$ (up to finite difference). 
\end{proof}

\begin{corollary} 
	$H$ is atomless hyperhypersimple. 
\end{corollary}

\section{Other Boolean algebras} 
\label{sec:other_algebras}

In this section, we explain how to modify the proof of \cref{thm:atomless} to show:

\begin{theorem} \label{thm:other_algebras}
Let $A$ be $\lowtwo$ and coinfinite. For any $\Sigma_3$-Boolean algebra~$B$ there is some c.e.\ $H\supseteq A$ such that $\+L^*(H)\cong B$. 
\end{theorem}

Let~$B$ be a $\Sigma_3$-Boolean algebra. This means that it is the quotient of the (standard computable copy of the) atomless Boolean algebra by a $\Sigma_3$ ideal. An equivalent characterisation is in terms of Boolean algebras generated by trees. Recall that for a tree $T\subseteq 2^{<\w}$ we defined the Boolean algebra $B(T)$, generated from the elements of~$T$ according to the rules: (i) every $\tau\in T$ is the join of its children; and (ii) if $\s,\tau\in T$ are incomparable then $\s\wedge \tau = 0_{B(T)}$. We noted that this implies that if $\tau$ is non-extendible on~$T$, then $\tau = 0_{B(T)}$ (and the atoms of $B(T)$ are $\tau$ where~$\tau$ isolates a unique path of~$T$). Note that there are other ways of producing Boolean algebras from trees, for example, ones in which leaves of the tree are atoms. However, the advantage of the representation that we chose is:
\begin{itemize}
  \item The $\Sigma_3$ Booleans algebras are precisely the algebras $B(T)$, where $T$ is $\Delta_3$. 
\end{itemize}
(If we used the other representation, we would need $\Sigma_3$ trees). The idea of the modified construction is to add nodes that decide, for each $\rho\in \{0,1\}^e$, whether $\rho\in T$ or not, again utilising $\Delta_3$ guessing. A child~$\beta$ of such~$\alpha$ that guesses that $\rho\notin T$ will attempt to enumerate all of $\Y{\alpha,\rho}$ into~$H$. The technique is identical to eliminating balls to decide $W_e$ on $\Y{\alpha,\rho}$.

\subsection{The construction}

Most of the construction is as above, so we explain the new ingredients. We now have four kinds of nodes:
\begin{itemize}
  \item $e$-tree nodes, of length $4e$; 
  \item $e$-decision nodes, of length $4e+1$; 
  \item parent and child $e$-splitting nodes, of lengths $4e+2$ and $4e+3$. 
\end{itemize}

For an $e$-tree node $\alpha$ and $E\subseteq \{0,1\}^e$, the statement $\psi(\alpha,E)$ is:
\[
  E = T\cap \{0,1\}^e. 
\]
This is a finite Boolean combination of $\Delta_3$ statements. Again the children are $\alpha\conc E_n$ for $E\subseteq \{0,1\}^e$ and $n<\w$. We define $\seqell{\alpha\conc E_n} = \seqell{\psi(\alpha,E),n)}$ as above. 

\begin{definition} \label{def:more_BAs:new_grabbable}
  Let $\beta = \alpha\conc E_n$ be a child of an $e$-tree node~$\alpha$, and let $\rho\in \{0,1\}^e$. We say that a ball~$x$ is \emph{pullable} by $(\beta,\rho)$  at a stage~$s$ if:
  \begin{orderedlist}
 \item $x\in \Y{\alpha,\rho}{s}\setminus \Y{\le\!\beta,\rho}{s}$; 
 \item $\rho\in E$; and 
 \item either:
 \begin{itemize}
   \item $|\Y{\beta,\rho}{s}|< \oneell{\beta}{s}$, or
   \item $\Y{\beta,\rho}{s}\ne \emptyset$ and $x< \max \Y{\beta,\rho}{s}$.
 \end{itemize}
  \end{orderedlist}
  We say that~$x$ is \emph{eliminable} by $(\beta,\rho)$ at stage~$s$ if:
  \begin{orderedlist}[resume]
  \item $x\in \Y{\alpha,\rho}{s}\setminus \Y{\le\!\beta,\rho}{s}$; 
  \item $\rho\notin E$; 
  \item $x\ne \min \Y{\alpha,\rho}{s}$; and
  \item $x<\oneell{\alpha}{s}$.
  \end{orderedlist}
\end{definition}

We need to modify \cref{def:atomless:certified}\ref{item:atomless:def_certified:ready:certified} so that we require certification only for those $\rho\in \{0,1\}^e$ that are guessed to be on the tree~$T$. If $\alpha$ is an $e$-decision, $e$-splitting, or $(e+1)$-tree node, then we let
\[
  E(\alpha)
\]
denote the set $E\subseteq \{0,1\}^e$ such that for some~$n$, $\epsilon\conc E_n\preceq \alpha$, where~$\epsilon$ is the $e$-tree node preceding~$\alpha$. We let $E(\lambda) = \{\lambda\}$. Then in the new version of \cref{def:atomless:certified}\ref{item:atomless:def_certified:ready:certified}, we require that the conditions hold for those $\rho\in E(\beta)$. 

The rest of the construction follows the atomless hyperhypersimple construction above, verbatim. 

\subsection{The verification}

For the verification, we only need to replace the second part of \cref{prop:atomless:main_true_path_lemma}. The new version of \cref{prop:atomless:main_true_path_lemma}(b) replaces $\{0,1\}^e$ by $E(\alpha)$:
\begin{itemize}
  \item If~$\alpha$ lies on the true path, then for all $\rho\in E(\alpha)$, for all~$k$, there is some~$s$ such that $|\Y{\alpha,\rho}{s}\cap C_s|\ge k$. 
\end{itemize}

The rest of the verification follows without change, showing that if $\rho\in T$ then $Z(\rho)$ is infinite, while if $\rho\notin T$ then $Z(\rho)$ is finite.

\bibliography{hhs}
\bibliographystyle{alpha}
\end{document}